\documentclass[11pt,reqno]{amsart}

\usepackage{amsfonts, amssymb, amscd}
\usepackage{amsmath}
\usepackage{verbatim}
\usepackage{amssymb}
\usepackage{mathrsfs}
\usepackage{graphicx}
\usepackage{cite}
\usepackage[all]{xy}
\usepackage{tikz}
\usepackage[toc,page]{appendix}
\usepackage{tikz-cd}

\usepackage{graphicx}
\usepackage{float}

\theoremstyle{definition}
\newtheorem{theorem}{Theorem}[section]
\newtheorem{lemma}[theorem]{Lemma}
\newtheorem{proposition}[theorem]{Proposition}
\newtheorem{definition}[theorem]{Definition}

\newtheorem{corollary}[theorem]{Corollary}
\newtheorem{remark}[theorem]{Remark}
\newtheorem{conjecture}[theorem]{Conjecture}
\numberwithin{equation}{section}

\begin{document}
\title{On duality of certain GKZ hypergeometric systems}

\begin{abstract}
We study a pair of conjectures on better behaved GKZ hypergeometric systems of PDEs inspired by Homological mirror symmetry for crepant resolutions of Gorenstein toric singularities. We prove the conjectures in the case of dimension two.
\end{abstract}

\author{Lev Borisov}
\address{Department of Mathematics\\
Rutgers University\\
Piscataway, NJ 08854} \email{borisov@math.rutgers.edu}

\author{Zengrui Han}
\address{School of Mathematical Sciences\\
University of Science and Technology of China\\
Hefei, Anhui, 230026} \email{hanzr@mail.ustc.edu.cn}

\author{Chengxi Wang}
\address{Department of Mathematics\\
Rutgers University\\
Piscataway, NJ 08854} \email{cw674@math.rutgers.edu}

\maketitle


\tableofcontents

\section{Introduction}\label{intro}
Homological Mirror Symmetry, proposed by Kontsevich in \cite{Kontsevich}, is a mathematical statement of equivalence of derived categories of coherent sheaves of complex manifolds with derived Fukaya categories of the mirror symplectic manifolds,  see \cite{Abouzaid, Sheridan}. In particular, it predicts that two manifolds with the same mirror must have equivalent derived categories of coherent sheaves.

\medskip
Moreover, physics predicts that these equivalences can be realized as some sort of parallel transports in isotrivial families of triangulated categories. As a consequence, when one passes to the complexified Grothendieck's groups,
one expects to have a flat vector bundle that interpolates from one manifold to another.
In  \cite{BorisovHorja} such vector bundle (or, rather, two bundles due to noncompactness issues) were constructed in the setting of stack resolutions of Gorenstein toric singularities. The two bundles should be related to each other by a duality which is reminiscent of Poincare duality between usual and compactly supported cohomology. The purpose of this paper is to verify the duality conjecture of  \cite{BorisovHorja} in the two-dimensional case.

\medskip
The setting of \cite{BorisovHorja} is as follows. Let $\Delta$ be a convex polytope in lattice $N_1$. We consider  $(\Delta,1)$ in $N=N_1\oplus {\mathbb Z}$ and the cone
$C={\mathbb R}_{\geq 0} (\Delta,1)$. Let $v_1,\ldots, v_n$ be lattice points of $(\Delta,1)$ that include all of its vertices. To these data, one can associate two systems of linear PDEs
on functions of $n$ variables $x_1,\ldots, x_n$, called $bbGKZ(C,0)$ and $bbGKZ(C^\circ,0)$ with solution spaces of dimension equal to the normalized volume of $\Delta$.

\medskip
In certain limits associated to simplicial subdivisions $\Sigma$ of $C$,
the results on \cite{BorisovHorja} give isomorphisms between the $K$-theory spaces of the toric Deligne-Mumford stacks and its dual space and the solutions of $bbGKZ(C,0)$ and $bbGKZ(C^{\circ},0)$ respectively in a neighborhood of the limit point. The isomorphism is  given by certain Gamma series $\Gamma$ and $\Gamma^\circ$. The following
Conjecture \ref{7.3} and Conjecture \ref{7.1} were formulated in \cite{BorisovHorja}. They describe the duality between the two systems and confirm the isotrivial family predictions at the level of complexified $K$-groups.
\begin{conjecture}\label{7.3}
\cite{BorisovHorja}
There exists a collection of polynomials $p_{c,d}(x_1,\ldots,x_n)$ indexed by $c\in C, d\in C^{\circ}$ such that
\begin{enumerate}
\item Only a finite number of $p_{c,d}$ are nonzero.
\item For any pair of solutions $(\Phi_c)$ of $bbGKZ(C,0)$ and $(\Psi_d)$ of $bbGKZ(C^{\circ},0)$, the sum
\begin{equation}\label{pcd}
\sum_{c,d}p_{c,d}\Phi_c\Psi_d
\end{equation}
 is constant as a function of $(x_1,\ldots,x_n)$.
\item The pairing given by \eqref{pcd} is non-degenerate.
\item For any projective simplicial subdivision $\Sigma$, the pairing given by \eqref{pcd} is the inverse of the Euler characteristics pairing between $K_0(\mathbb{P}_{\Sigma})$ and $K^c_0(\mathbb{P}_{\Sigma})$ under the $\Gamma$ and $\Gamma^{\circ}$.
\end{enumerate}
\end{conjecture}

\begin{conjecture}\label{7.1}
\cite{BorisovHorja} There are commutative diagrams of isomorphisms as follows

\smallskip

\hskip 86 pt
\xymatrix{
K_0(\mathbb{P}_{\Sigma_2})^{\vee} \ar[d]_{\Gamma} \ar[r]^{pp^{\vee}}
& K_0(\mathbb{P}_{\Sigma_1})^{\vee} \ar[d]_{\Gamma}\\
bbGKZ(C,0) \ar[r]^{a.c.}
& bbGKZ(C,0)
}

\smallskip

\hskip 82 pt
\xymatrix{
K^c_0(\mathbb{P}_{\Sigma_2})^{\vee} \ar[d]_{\Gamma^{\circ}} \ar[r]^{pp^{\vee}}
& K^c_0(\mathbb{P}_{\Sigma_1})^{\vee} \ar[d]_{\Gamma^{\circ}}\\
bbGKZ(C^{\circ},0) \ar[r]^{a.c.}
& bbGKZ(C^{\circ},0)
}

\smallskip\noindent
where the top rows are the duals of the isomorphisms induced by the pullback-pushforward derived functors, and the bottom rows are analytic continuous along a certain path in the domain of parameters considered in \cite{BorisovHorja2}.
\end{conjecture}

In this paper, we prove these two conjectures in the case of $\mathrm{rk \, N}=2$ by an explicit calculation. The key to it is a nice guess of the polynomials $p_{c,d}$.

\medskip
The paper is organized as follows. In Section \ref{bbGKZ}, we  review the better behaved GKZ hypergeometric systems and the solutions to these systems in the form of contour integrals in the case of $\mathrm{rk \, N}=2$. In Section \ref{pairingofsol}, we find a pairing which satisfies $(1),(2),(3)$ in Conjecture \ref{7.3}. In Section \ref{pairingHH^c}, we define a pairing between the cohomology ring and its dual which is compatible with the Euler characteristics pairing between $K_0(\mathbb{P}_{\Sigma})$ and $K^c_0(\mathbb{P}_{\Sigma})$ defined in \cite{BorisovHorja}. Then we compute the inverse of this pairing. Section \ref{pluginGamma} contains the calculation of the pairing defined in Section \ref{pairingofsol} on the Gamma series solutions to the better-behaved GKZ systems. Then we prove part $(4)$ of Conjecture \ref{7.3} and Conjecture \ref{7.1} in the case of $\mathrm{rk \, N}=2$. Section \ref{futuredire} briefly describes further directions of research.

\medskip

\noindent{\it Acknowledgements.} L.B. and C.W. were partially supported by the NSF Grant DMS-1601907.

\section{Better behaved GKZ hypergeometric systems}\label{bbGKZ}
In this section, we give an overview of the so-called better behaved GKZ hypergeometric systems which were defined in \cite{BorisovHorja1}. When the rank of the lattice $N$ in the definition of the systems  is two,  we show that a solution to the system can be given in the form of contour integrals which will be further discussed in Section \ref{pairingofsol}. Then we prove $(1),(2),(3)$ in Conjecture \ref{7.3} in the case of $\mathrm{rk \, N}=2$.

\smallskip

Let $C$ be a finite rational polyhedral cone in a lattice $N=N_1\oplus {\mathbf Z}$ and based on a lattice polytope $(\Delta,1)$ and let $\{v_i\}_{i=0}^n$ be a set of $n+1$ elements of $C$ that includes all vertices of $(\Delta,1)$.

\begin{definition}\label{definebbGKZ}
Let $\{\Phi_c(x_0,\cdots,x_n)\}$ be an infinite collection of $(n+1)$-variable functions which are indexed by lattice points $c\in C$. We consider the system of partial differential equations on these functions as follows:
\begin{equation}\label{GKZ}
    \partial_i\Phi_c=\Phi_{c+v_i},\ \ \ \ \sum_{i=0}^n\mu(v_i)x_i\partial_i\Phi_c+\mu(c)\Phi_c=0
\end{equation}
for all linear functions $\mu\in N^{\vee}$ and $c\in C$. We call this system $bbGKZ(C,0)$. Similarly, we define $bbGKZ(C^{\circ},0)$ by considering lattice points in the interior of $C$. In the following sections we denote the solutions to $bbGKZ(C^{\circ},0)$ by $\Psi_d$.
\end{definition}

\begin{remark}\label{remdimn}
It was shown in \cite{BorisovHorja1} that the dimension of the solution spaces to $bbGKZ(C,0)$ and $bbGKZ(C^{\circ},0)$ is equal to the normalized volume of $C$, provided that $f=\Sigma_{i=0}^{n}x_i[v_i]$ is $\Delta-$nondegenerate in the sense of Batyrev \cite{Batyrev}.
\end{remark}

Now we focus on the special case of $\mathrm{rk \,N}=2$. Let $C$ be the cone with two rays passing through $(0,1)$ and $(n,1)$. Let $v_i=(i,1),0\leq i\leq n$ be the successive lattice points lying on the line of degree one. We now construct solutions to the $bbGKZ$ systems using contour integrals.

\smallskip
Define  the polynomial of one complex variable $z$  by $f(z)=x_0+x_1z+\cdots+x_nz^n$.
It is easy to see that the $\Delta-$nondegeneracy in this case  translates to $x_0\neq0$, $x_n\neq 0$ and roots of $f(z)$ being distinct. Equivalently, $f$ has exactly $n$ distinct nonzero roots which we will denote by $\xi_i$. These roots depend analytically on $(x_i)$.

\smallskip
First, we consider the solutions to $bbGKZ(C,0)$. Note that $\Phi_{(0,0)}=1$ and $\Phi_c=0$ for nonzero lattice point $c\in C$ form a solution to $bbGKZ(C,0)$ which we denote by $\Phi^0$.

\smallskip
To construct the other solutions,  cut $\mathbb{CP}^1$ by any path $S$ from $0$ to $\infty$ which avoids $\xi_i$ and pick a branch of $\log z$. Then the following formulas provide solutions to $bbGKZ(C,0)$ for any (homotopy class of) closed smooth path $\gamma$  in $\mathbb{C}$ disjoint from $\xi_i$ and the cut $S$.
\begin{definition}\label{phi}
We define functions
\begin{equation*}
\begin{split}
    \Phi_{(0,0)}^{\gamma}&=\frac{1}{2\pi \rm i}\oint_{\gamma}\log z\frac{f^{\prime}(z)}{f(z)}dz+\frac{1}{n}c(\gamma)(\log x_n-\log x_0) 
\end{split}
\end{equation*}
where $c(\gamma) = \frac 1{2\pi\rm i} \int_{\gamma}\frac {f'(z)}{f(z)}\,dz$ denotes the winding number of $f(z)$ around $0$ as $z$ moves along $\gamma$. For $c=(k,l)\in C^\circ$ we define
\begin{equation*}
    \Phi_{(k,l)}^{\gamma}=\frac{1}{2\pi \rm i}\oint_{\gamma}\frac{(-1)^l(l-1)!z^k}{f(z)^l}\frac{dz}{z},\ \ \ \text{for all }(k,l)\in C^{\circ}
\end{equation*}
and we define
\begin{equation*}
\begin{split}
    \Phi_{(0,l)}^{\gamma}=\frac{1}{2\pi \rm i}\oint_{\gamma}\frac{(-1)^l(l-1)!}{f(z)^l}\frac{dz}{z}
    +\frac {c(\gamma)}n\frac{(-1)^l(l-1)!}{x_0^l}, \\
    \Phi_{(nl,l)}^{\gamma}=\frac{1}{2\pi \rm i}\oint_{\gamma}\frac{(-1)^l(l-1)!z^{nl}}{f(z)}\frac{dz}{z}
     -\frac {c(\gamma)}n\frac{(-1)^l(l-1)!}{x_n^l}.
\end{split}
\end{equation*}
We collect these functions into $\Phi^{\gamma}=(\Phi_{(k,l)}^{\gamma})_{(k,l)\in C}$. It is defined up to a multiple of $\Phi^0$ which comes from a choice of branches of the logarithm.
\end{definition}

\begin{lemma}
The functions $\Phi^{\gamma}=(\Phi_{(k,l)}^{\gamma})_{(k,l)\in C}$ defined in Definition \ref{phi} form a solution to $bbGKZ(C,0)$.
\end{lemma}
\begin{proof}
Let $\delta_{kn}$, $\delta_{k0}$ be the Kronecker symbols.
We have
\begin{equation*}
\begin{split}
\partial_k\Phi_{(0,0)}^{\gamma}&=\frac{1}{2 \pi \rm i} \oint_{\gamma}(\log z)\, \partial_k(\frac{f'(z)}{f(z)})dz+ \frac{c(\gamma)}{n}(\frac{\delta_{kn}}{x_n}-\frac{\delta_{k0}}{x_0})\\
&=\frac{1}{2 \pi \rm i} \oint_{\gamma}(\log z)\, (\partial_k\log(f(z)))'dz+ \frac{c(\gamma)}{n}(\frac{\delta_{kn}}{x_n}-\frac{\delta_{k0}}{x_0})\\
&=-\frac{1}{2 \pi \rm i}\oint_{\gamma}\partial_k(\log(f(z)))\frac{1}{z}dz+\frac{c(\gamma)}{n}(\frac{\delta_{kn}}{x_n}-\frac{\delta_{k0}}{x_0})\\
&=-\frac{1}{2 \pi \rm i}\oint_{\gamma}\frac{z^k}{f(z)}\frac{1}{z}dz+\frac{c(\gamma)}{n}(\frac{\delta_{kn}}{x_n}-\frac{\delta_{k0}}{x_0})=\Phi_{(k,1)}^{\gamma}.\\
\end{split}
\end{equation*}
For $0<k<nl$, we get
\begin{equation*}
\begin{split}
\partial_j\Phi_{(k,l)}^{\gamma}&=\frac{1}{2 \pi \rm i} \oint_{\gamma}\partial_j(\frac{(-1)^l(l-1)!z^k}{f(z)^lz})dz\\
&=\frac{1}{2 \pi \rm i} \oint_{\gamma}\frac{(-l)(-1)^l(l-1)!z^k \partial_j(f(z))}{f(z)^{l+1}z}dz\\
&=\frac{1}{2 \pi \rm i} \oint_{\gamma}\frac{(-1)^{l+1}\,l!\,z^{k+j}}{f(z)^{l+1}z}dz=\Phi^{\gamma}_{(k+j,l+1)}.
\end{split}
\end{equation*}
For $k\in\{0,nl\}$, the calculation is analogous and is left to the reader.

\smallskip

In view of the above, it suffices to verify the torus invariance of $\Phi_{(0,0)}^{\gamma}$, i.e.,
$$\sum_{j=0}^nx_j\partial_j \Phi_{(0,0)}^{\gamma}=0=\sum_{j=0}^{n}jx_j\partial_j \Phi_{(0,0)}^{\gamma}.$$
We have
\begin{equation*}
\begin{split}
\sum_{j=0}^nx_j\partial_j \Phi_{(0,0)}^{\gamma}&=-\frac{1}{2 \pi \rm i} \oint_{\gamma}\frac{\sum_{j=0}^nx_jz^j}{f(z)z}dz+\frac{c(\gamma)}{n}(\frac{x_n}{x_n}-\frac{x_0}{x_0})\\
&=-\frac{1}{2 \pi \rm i} \oint_{\gamma}\frac{1}{z}dz=0
\end{split}
\end{equation*}
because $\gamma$ is disjoint from the cut.
We have
\begin{equation*}
\begin{split}
\sum_{j=0}^njx_j\partial_j \Phi_{(0,0)}^{\gamma}&=-\frac{1}{2 \pi \rm i} \oint_{\gamma}\frac{\sum_{j=0}^njx_jz^{j-1}}{f(z)}dz+\frac{c(\gamma)}{n}(\frac{nx_n}{x_n})\\
&=-\frac{1}{2 \pi \rm i} \oint_{\gamma}\frac{f'(z)}{f(z)}dz+c(\gamma)=0.
\end{split}
\end{equation*}
\end{proof}

\begin{corollary}\label{cor2.5}
For a small loop $\gamma$ around $\xi=\xi_i$ for some $i$, we have a solution of $bbGKZ(C,0)$ with
\begin{equation*}
\Phi_{(0,0)}=\log\xi+\frac{1}{n}\log x_n-\frac{1}{n}\log x_0,
\end{equation*}
\begin{equation*}\label{Phik1}
    \Phi_{(k,1)}=-\frac{\xi^k}{f^{\prime}(\xi)\xi},\ \ 0<k<n,
\end{equation*}
\begin{equation*}\label{Phi01}
    \Phi_{(0,1)}=-\frac{1}{f^{\prime}(\xi)\xi}-\frac{1}{nx_0},\ \ \
    \Phi_{(n,1)}=-\frac{\xi^n}{f^{\prime}(\xi)\xi}+\frac{1}{nx_n}.
\end{equation*}
\end{corollary}

\begin{proof}
Follows from Cauchy's integral formula.
\end{proof}

\begin{remark}
The $bbGKZ(C,0)$ system in this case is just the usual $GKZ$ system for $v_0=(0,1),\ldots,v_n=(0,n)$. So the above formulas must be well-known, but we didn't find them in the literature. The $bbGKZ(C^{\circ},0)$ statements below are likely new, since this system is not as well-studied.
\end{remark}

Now we consider the solution to $bbGKZ(C^{\circ},0)$. It will be convenient for our future calculations to use a different dummy variable $w$. Let $f(w)=x_0+x_1w+\cdots+x_nw^n$ be a polynomial of complex variable $w$. We consider (smooth) contours $\lambda\in \mathbb{CP}^1\setminus \{0,\infty,\xi_1,\ldots,\xi_n\}$ whose boundary may contain $0, \infty$, see $(1)$ of Figure \ref{fig:1}. Let $\lambda_0$ be the path from $0$ to $\infty$ along the negative real axis, perturbed to avoid any negative real $\xi_i$, see $(2)$ of Figure \ref{fig:1}. Let $\lambda_i=\gamma_i, i=1,\ldots,n$ which are small circles centered at the roots $\xi_i$ of $f(w)$.

\begin{figure}[H]
  \includegraphics[width=1.0\textwidth]{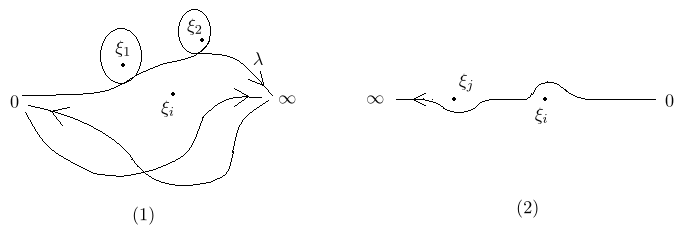}
  \caption{}
  \label{fig:1}
\end{figure}

\begin{definition}\label{psi}
We define functions
\begin{equation*}
    \Psi_{(k,l)}=\frac{1}{2\pi \rm i}\int_{\lambda}\frac{w^k(-1)^l(l-1)!}{f(w)^l}\frac{dw}{w},\ \ \ \ \text{for all }(k,l)\in C^{\circ}.
\end{equation*}
\end{definition}
\begin{lemma}
The functions defined in Definition \ref{psi} form a solution to $bbGKZ(C^{\circ},0)$.
\end{lemma}
\begin{proof}
Convergence at $0$ and $\infty$ follows from $(k,l)\in C^{\circ}$, i.e., $0 < k < nl$.
It's easy to verify the first equation in \eqref{GKZ}. To show the second equation in \eqref{GKZ}, we have
\begin{equation*}
\begin{split}
    \sum_{j}x_j\partial_j\Psi_{(k,l)}+l\Psi_{(k,l)}&=\frac{1}{2\pi \rm i}\int_{\lambda}(\frac{\sum_jx_jw^{k+j}(-1)^{l+1}l!}{f(w)^{l+1}}+\frac{l!(-1)^l w^k}{f(w)^l})\frac{dw}{w}\\
    &=\frac{1}{2\pi \rm i}\int_{\lambda}0 \frac{dw}{w}=0
    \end{split}
\end{equation*}
and
\begin{equation*}
    \sum_{j}j x_j\partial_j\Psi_{(k,l)}+k\Psi_{(k,l)}=\frac{1}{2\pi \rm i}\int_{\lambda}(\frac{\sum_j j x_j w^{k+j}(-1)^{l+1}l!}{f(w)^{l+1}}+\frac{l!(-1)^l k w^k}{f(w)^l})\frac{dw}{w}=0
\end{equation*}
due to
\begin{equation*}
    \frac{\sum_j j x_j w^{k+j}(-1)^{l+1}l!}{f(w)^{l+1}}+\frac{l!(-1)^l k w^k}{f(w)^l}=\frac{d}{dw}\big(\frac{l!(-1)^l w^k}{f(w)^l}\big)
\end{equation*}
and $\frac{l!(-1)^l w^k}{f(w)^l} \rightarrow 0$ as $w \rightarrow 0, \infty$.
\end{proof}

We denote these solutions by $\Psi^{\lambda}=(\Psi^{\lambda}_{(k,l)})_{(k,l)\in C^{\circ}}$. We consider the solutions $\Psi^{\lambda_0}$ and $\Psi^{\lambda_1},\ldots,\Psi^{\lambda_n}$, where $\lambda_i=\gamma_i$ for $i=1,\ldots,n$.

\begin{remark}
In Section \ref{pairingofsol}, we will show that the solutions $\{\Phi^0,\Phi^{\gamma_1},\ldots,\Phi^{\gamma_n}\}$ and $\{\Psi^{\lambda_0},\Psi^{\gamma_1},\ldots,\Psi^{\gamma_n}\}$ constructed in Definition \ref{phi} and Definition \ref{psi} generate the space of solutions to $bbGKZ(C,0)$ and $bbGKZ(C^{\circ},0)$ respectively. They satisfy relations $\Phi^{\gamma_1}+\ldots+\Phi^{\gamma_n}=\alpha \Phi^0$ and $\Psi^{\gamma_1}+\ldots+\Psi^{\gamma_n}=0$. These constructions will be used to prove our main result in Section \ref{pairingofsol}.
\end{remark}

\section{Pairing of solutions in the case of $\mathrm{rk \,N}=2$}\label{pairingofsol}
In this section, we define a non-degenerate pairing between the solutions of $bbGKZ(C,0)$ and $bbGKZ(C^{\circ},0)$ for $\mathrm{rk \, N}=2$ which satisfies $(1),(2)$ in Conjecture \ref{7.3}.
\smallskip

Let $C$ be a cone in $N=\mathbb{Z}^2$ with two rays passing through $(0,1)$ and $(n,1)$. Let
\begin{equation*}
    v_0=(0,1),\ v_1=(1,1),\cdots,\ v_n=(n,1)
\end{equation*}be the successive lattice points lying on the line of degree one.
For $i<j$, we denote the cone generated by $v_i$ and $v_j$ by $\sigma_{ij}$. We define the following pairing.
\begin{definition}\label{pairing}
For any pair of solutions $\Phi=(\Phi_c)$ and $\Psi=(\Psi_d)$ of $bbGKZ(C,0)$ and $bbGKZ(C^{\circ},0)$ respectively, we define a pairing
\begin{equation}
\begin{split}
    \langle\Phi,\Psi\rangle=&\Phi_{(0,0)}\sum_{0\leq i<j\leq n}(j-i)^2x_ix_j\Psi_{v_i+v_j}\\
    &-\sum_{0\leq i<j\leq n}n(j-i)x_ix_j\sum_{\substack{c+d=v_i+v_j \\ c,d\in\sigma_{ij} \\ \deg{c}=1}}\delta_{ij}^{c}\Phi_c\Psi_d,
\end{split}
\end{equation}
where $\delta_{ij}^c=1$ if $c$ lies in the interior of the cone $\sigma_{ij}$ and $\delta_{ij}^c=\frac{1}{2}$ otherwise. In the second summation, we only consider the terms with $d$ in $C^{\circ}$.
\end{definition}

As defined, $\langle\Phi,\Psi\rangle$ is a function of $x_0,\ldots,x_n$. The main result of this section is the following.

\begin{theorem}\label{pairingthm}
For any pair of solutions $(\Phi_c)$ and $(\Psi_d)$ of $bbGKZ(C,0)$ and $bbGKZ(C^{\circ},0)$ respectively, the pairing defined in Definition \ref{pairing} is a constant.
\end{theorem}

\begin{proof}
We will calculate the pairing \eqref{pairing} for the solutions $\{\Phi^0,\Phi^{\gamma_1},\ldots,\Phi^{\gamma_n}\}$ and $\{\Psi^{\lambda_0},\Psi^{\lambda_1},\ldots,\Psi^{\lambda_n}\}$ constructed in Definition \ref{phi} and Definition \ref{psi} respectively, where $\lambda_l=\gamma_l$ for $l=1,\ldots,n$. We check by substituting the solutions into \eqref{pairing}.

\smallskip

We first calculate the terms with $\Phi_{(0,0)}^{\gamma_k}$. We have
\begin{equation*}
    \sum_{0\leq i<j\leq n}(j-i)^2x_ix_j\Psi_{v_i+v_j}^{\lambda_l}=\frac{1}{2\pi \rm i}\int_{\lambda}\sum_{0\leq i\leq j\leq n}\frac{(j-i)^2x_ix_jw^{i+j}}{f(w)^2}\frac{dw}{w}.
\end{equation*}
We note the summation
\begin{equation*}
    \sum_{0\leq i\leq j\leq n}(j-i)^2x_ix_jw^{i+j}=\frac{1}{2\pi \rm i}\big(f(w)(w\partial_w)^2f(w)-(w\partial_wf(w))^2\big).
\end{equation*}
Hence we have
\begin{equation*}
\begin{split}
    \sum_{0\leq i<j\leq n}(j-i)^2x_ix_j\Psi_{v_i+v_j}^{\lambda_l}&=\frac{1}{2\pi \rm i}\int_{\lambda}\frac{f(w)(w\partial_w)^2f(w)-(w\partial_wf(w))^2}{f(w)^2}\frac{dw}{w}\\
    &=\frac{1}{2\pi \rm i}\int_{\lambda}\frac{d}{dw}(\frac{w\partial_{w}f(w)}{f(w)})dw.
\end{split}
\end{equation*}
If $\lambda_l=\gamma_l$ for $l=1,\ldots,n$, then this integral is $0$. If $\lambda_l=\lambda_0$ is a path from $0$ to $\infty$, then this integral equals $$\frac{1}{2\pi \rm i}\big(\lim_{w\rightarrow\infty}\frac{wf^{\prime}(w)}{f(w)}-\lim_{w\rightarrow 0}\frac{wf^{\prime}(w)}{f(w)}\big)=\frac{1}{2\pi \rm i}(n-0)=\frac{n}{2\pi \rm i}.$$
So by Corollary \ref{cor2.5} the terms with $\Phi_{(0,0)}^{\gamma_k}$ are
\begin{equation}\label{phi00term}
    \Phi_{(0,0)}^{\gamma_k}\sum_{0\leq i<j\leq n}(j-i)^2x_ix_j\Psi^{\lambda_l}_{v_i+v_j}=\frac{1}{2\pi \rm i}(n\log\xi+\log x_n-\log x_0)\delta_{l}^{0},
\end{equation}
where $\delta_{l}^{0}=1$ if $l=0$ and $\delta_{l}^{0}=0$ if $l\neq 0$.

\smallskip
Then we calculate the other terms. To simplify the calculation, we first ignore the extra terms $\frac{1}{nx_0}$ and $\frac{1}{nx_n}$ in the formula for $\Phi_{(0,1)}^{\gamma_k}$ and $\Phi_{(n,1)}^{\gamma_k}$ in Corollary \ref{cor2.5} and divide the summation into four parts:
\begin{equation*}
    -\sum_{0\leq i<j\leq n}n(j-i)x_ix_j\sum_{\substack{c+d=v_i+v_j \\ c,d\in\sigma_{ij} \\ \deg{c}=1}}\delta_{ij}^{c}\Phi^{\gamma_k}_c\Psi^{\lambda_l}_d=(\frac{1}{2\pi \rm i})^2\int_{\gamma}\frac{(G_1+G_2+G_3+G_4)\,dw}{f(w)wf^{\prime}(\xi)\xi}
\end{equation*}
where $G_1$ has terms with $i\not=0,j\not=n$, $G_2$ has terms with $i\not=0,j=n$, $G_3$ has terms with $i=0,j\not=n$, $G_4$ has terms with $i=0,j=n$. We calculate them separately as follows:
\begin{equation}\label{otherterms-1}
\begin{split}
    G_1=&-n\sum_{0<i<j<n}(j-i)x_ix_j(\sum_{i<c<j}\xi^c w^{i+j-c}+\frac{1}{2}\xi^i w^{j}+\frac{1}{2}\xi^j w^{i})\\
    =&-n\sum_{0<i<j<n}(j-i)x_ix_j\frac{1}{2}\frac{(w+\xi)}{(w-\xi)}(\xi^i w^j-\xi^j w^i)\\
    =&-\frac{n(w+\xi)}{2(w-\xi)}\frac{1}{2}\sum_{0<i,j<n}(j-i)x_ix_j(\xi^i w^j-\xi^j w^i)\\
    =&-\frac{n(w+\xi)}{2(w-\xi)}\sum_{0<i,j<n}(j x_i x_j \xi^i w^j-i x_i x_j \xi^i w^j)\\
    =&-\frac{n(w+\xi)}{2(w-\xi)}\big((-x_0-x_n\xi^n)(wf^{\prime}(w)-nx_nw^n)\\
    &\qquad\qquad\qquad\qquad -(\xi f^{\prime}(\xi)-nx_n\xi^n)(f(w)-x_nw^n-x_0)\big).\\
\end{split}
\end{equation}
Similarly, we get
\begin{equation}\label{otherterms-2}
\begin{split}
    G_2=&-\frac{n(w+\xi)}{2(w-\xi)}x_n(-\xi f^{\prime}(\xi)w^n-n\xi^n f(w)+\xi^n wf^{\prime}(w))\\
    &+\frac{n^2}{2}x_0x_n\frac{w+\xi}{w-\xi}(w^n-\xi^n)+\frac{nx_nw^n}{2}(-\xi f^{\prime}(\xi)-nx_0),\\
    G_3=&-\frac{n(w+\xi)}{2(w-\xi)}x_0(wf^{\prime}(w)-\xi f^{\prime}(\xi))+\frac{n^2(w+\xi)}{2(w-\xi)}x_0x_n(w^n-\xi^n)\\
    &+\frac{n}{2}x_0\xi f^{\prime}(\xi)-\frac{n^2}{2}\xi^nx_0x_n,\\
    G_4=&-n^2x_0x_n\frac{w+\xi}{2(w-\xi)}(w^n-\xi^n)+\frac{n^2}{2}x_0x_n(\xi^n+w^n).
\end{split}
\end{equation}

The contributions of the extra terms of $\Phi^{\gamma_k}_{(0,1)}$ and $\Phi^{\gamma_k}_{(n,1)}$ can be calculated as follows:
\begin{equation}\label{extraterm}
\begin{split}
    &-\xi f^{\prime}(\xi)\frac{1}{2}\sum_{0<j<n}jx_jw^j+\frac{1}{2}\xi f^{\prime}(\xi)\sum_{0<i<n}(n-i)x_iw^i\\
    =&-\frac{1}{2}\xi f^{\prime}(\xi)(2wf^{\prime}(w)-nf(w)-nx_nw^n+nx_0).
\end{split}
\end{equation}
Now we add up \eqref{phi00term}, \eqref{otherterms-1}, \eqref{otherterms-2}, and \eqref{extraterm}. Most of the terms cancel to give
\begin{equation}\label{pairkl}
\begin{split}
    \langle\Phi^{\gamma_k},\Psi^{\lambda_l}\rangle&=\ \frac{1}{2\pi \rm i}(n\log\xi+\log x_n-\log x_0)\delta_{l}^{0}\\
    &+\frac{1}{2\pi \rm i}\int_{\lambda}(\frac{n(w+\xi)}{2(w-\xi)}f(w)-wf^{\prime}(w)+\frac{n}{2}f(w))\frac{dw}{wf(w)}\\
    &=\ \frac{1}{2\pi \rm i}(n\log\xi+\log x_n-\log x_0)\delta_{l}^{0}\\
    &+\frac{1}{2\pi \rm i}\int_{\lambda}(-\frac{f^{\prime}(w)}{f(w)}+\frac{n(w+\xi)}{2(w-\xi)w}+\frac{n}{2w})dw.\\
 \end{split}
\end{equation}
If $\lambda_l=\gamma_l$ for $l=1,\ldots,n$, the pairing \eqref{pairkl} equals
\begin{equation*}
0+\frac{1}{2\pi \rm i}\int_{\lambda}(-\frac{f^{\prime}(w)}{f(w)}+\frac{n}{w-\xi})dw=n\delta_{kl}-1.
\end{equation*}
If $\lambda_l=\lambda_0$, the pairing \eqref{pairkl} equals
\begin{equation*}
\begin{split}
&\ \frac{1}{2\pi \rm i}(n\log\xi+\log x_n-\log x_0)+\frac{1}{2\pi \rm i}\int_{\lambda}(-\frac{f^{\prime}(w)}{f(w)}+\frac{n}{w-\xi})dw\\
    &=\ \frac{1}{2\pi \rm i}(n\log\xi+\log x_n-\log x_0)+\frac{1}{2\pi \rm i}(\mathrm{Log}\frac{(w-\xi)^n}{f(w)})\Big|_{0}^{\infty}\\
    &=\ \frac{1}{2\pi \rm i}(n\log\xi+\log x_n-\log x_0)+\frac{1}{2\pi \rm i}\big(\mathrm{Log}(\frac{1}{x_n})-(\mathrm{Log}(-\xi)^n-\mathrm{Log}x_0)\big)\\
    &=\ constant.
\end{split}
\end{equation*}

\end{proof}

\begin{corollary}\label{nondegen}
The pairing given in Definition \ref{pairing} is non-degenerate.
\end{corollary}
\begin{proof}
We calculate the pairing between two collections of solutions $$\{\Phi^0,\Phi^{\gamma_1},\ldots,\Phi^{\gamma_n}\}, \{\Psi^{\lambda_0},\Psi^{\lambda_1},\ldots,\Psi^{\lambda_n}\}$$ constructed in Definition \ref{phi} and Definition \ref{psi}. It is easy to check $\langle \Phi^0, \Psi^{\lambda_0}\rangle=\frac{n}{2\pi \rm i}$ and $\langle \Phi^0, \Psi^{\lambda_l}\rangle=0$ for $l=1,\ldots,n$. Then the matrix $P$ whose entries are the pairings between the two collections of solutions is given by
$$P=\begin{pmatrix}
   \frac{n}{2\pi \rm i} & 0 & 0 & 0 & \ldots & 0  \\
   * & n-1 & -1 & -1 & \ldots & -1\\
   * & -1 & n-1 & -1 & \ldots & -1 \\
   \vdots & \vdots & \vdots & \vdots & \ldots &\vdots \\
   * & -1 & -1 & -1 & \ldots &  n-1\\
\end{pmatrix}.$$
This matrix $P$ has rank $n$ since the $n\times n$ submatrix at lower right corner has rank $n-1$. Therefore, the pairing is non-degenerate and the dimension of the solution subspace spanned by
$\{\Phi^0,\Phi^{\gamma_1},\ldots,\Phi^{\gamma_n}\}$ is $n$ which equals the dimension of the solution space of $bbGKZ(C,0)$ by Remark \ref{remdimn}. Thus $\{\Phi^0,\Phi^{\gamma_1},\ldots,\Phi^{\gamma_n}\}$ generate the solution space of $bbGKZ(C,0)$. Similarly, $\{\Psi^{\lambda_0},\Psi^{\lambda_1},\ldots,\Psi^{\lambda_n}\}$ generate the solution space of $bbGKZ(C^{\circ},0)$.
\end{proof}

\begin{remark}
Theorem \ref{pairingthm} and Corollary \ref{nondegen} is a special case of the Conjecture 7.3 in \cite{BorisovHorja} when the dimension is two and with no gaps among $\{v_i\}$, i.e, $N=\mathbb{Z}^2$ and $\{v_i\}$ are successive lattice points lying on the line of degree one. For the case when there are some gaps, we set $x_i$ in \eqref{pairing} to be zero if there is a gap at $(i,1)$. The proof goes right through since the polynomial $f$ just misses some terms. Hence Theorem \ref{pairingthm} and Corollary \ref{nondegen} are valid for all cases of dimension two.
\end{remark}

\section{Pairing between $H$ and $H^c$}\label{pairingHH^c}
In this section, we define a pairing $\chi_{H}$ between the spaces  $H$ and $H^c$ defined in \cite{BorisovHorja} which is compatible with the pairing between $K_0(\mathbb{P}_{\Sigma})$ and $K^c_0(\mathbb{P}_{\Sigma})$. Then we calculate the inverse of $\chi_{H}\in (H\otimes H^c)^{\vee}$ which we denote by $\chi_{H}^{-1}\in H\otimes H^c$.

\smallskip
We first give a brief review of smooth toric DM stacks and their Grothendieck groups, see \cite{BorisovHorja1,BCS}.
\begin{definition}\cite{BCS}
Let $\Sigma$ be a simplicial fan in the lattice $N$ with lattice points $\{v_i\}_{i=1}^{n}$ on its rays (see Remark \ref{exfigure2}). We consider the open subset $U$ of $\mathbb{C}^n$ given by
	\begin{equation*}
	U:=\{(z_1,\cdots,z_n),\ \text{such that }\{i,\ z_i=0\}\in\Sigma\}.
	\end{equation*}
	Let the group $G$ be the subgroup of $(\mathbb{C^*})^{n}$ given by
	\begin{equation*}
	G=\{(\lambda_1,\cdots,\lambda_n),\ \text{such that }\prod_{i=1}^n\lambda_i^{\langle m,v_i \rangle}=1\text{ for all }m\in N^{\vee}\}.
	\end{equation*}
	The stack $\mathbb{P}_{\Sigma}$ is defined to be the stack quotient of $U$ by $G$.
\end{definition}

\begin{remark}\label{exfigure2}
We allow $\Sigma$ to be supported on a proper subset of the indices as in \cite{Ji}. For example, we could have $v_0=(0,1),\ldots,v_4=(4,1)$ and maximum cones of $\Sigma$ be $\mathbb{R}_{\geq 0}(v_0,v_1)$, $\mathbb{R}_{\geq 0}(v_1,v_4)$, see Figure \ref{sigma1}. In general, we let $\Sigma_{(1)}=\{v_{i_0},v_{i_1},\ldots,v_{i_r}, v_{i_{r+1}}\}$ be the set of rays in $\Sigma$, where $0=i_0<i_1<\ldots<i_{r+1}=n$ and $\#\Sigma_{(1)}=r+2$.
\end{remark}

\begin{figure}[H]
  \includegraphics[width=0.36\textwidth]{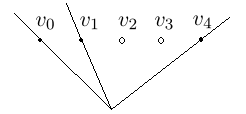}
  \caption{}
  \label{sigma1}
\end{figure}

\begin{definition}
	For each cone $\sigma\in\Sigma$, we denote $\mathrm{Box}(\Sigma)$ to be the set of points of $\gamma\in N$ that can be written as $\gamma=\sum_{i\in\sigma}\gamma_i v_i$ with $0\leq\gamma_i<1$. We denote by $\mathrm{Box}(\Sigma)$ the union of $\mathrm{Box}(\sigma)$ for all $\sigma\in\Sigma$.
\end{definition}
To each point in $\mathrm{Box}(\Sigma)$ we associate a twisted sector defined as the closed toric substack associated to the minimal cone $\sigma(\gamma)$ in $\Sigma$ which contains $\gamma$. We denote the corresponding cohomology by $H_{\gamma}$. The presentations for $H_{\gamma}$ and its dual module $H_{\gamma}^c$ are given in \cite{BorisovHorja} as follows.

\begin{definition}\label{oldH} \cite{BorisovHorja}
The cohomology $H_{\gamma}$ is defined to be the quotient of $\mathbb{C}[\overline{D}_i,i\in \mathrm{Star}(\sigma(\gamma))-\sigma(\gamma)]$ by the ideal generated by
\begin{equation*}
\begin{split}
&\prod_{i\in J}\overline{D}_i\text{ for all }J\not\in\mathrm{Star}(\sigma(\gamma)),\\
\sum_{i\in \mathrm{Star}(\sigma(\gamma))-\sigma(\gamma)}&(m\cdot v_i)\overline{D}_i \text{ for } m\in \mathrm{Ann}(v_i, i\in \sigma(\gamma)).
\end{split}
\end{equation*}
\end{definition}

\begin{definition}\label{oldH^c} \cite{BorisovHorja}
The module $H^c_{\gamma}$ over $\mathbb{C}[\overline{D}_i,i\in \mathrm{Star}(\sigma(\gamma))-\sigma(\gamma)]$ is generated by $F_I$ for $I\in \mathrm{Star}(\sigma(\gamma))$ and $\sigma_I^{\circ}\subseteq C^{\circ}$ with relations
\begin{equation*}
\begin{split}
&\overline{D}_iF_I-F_{I\cup \{i\}} \text{ for } i\notin I, I\cup \{i\} \in \mathrm{Star}(\sigma(\gamma)),\\
\text{and } &\overline{D}_iF_I \text{ for } i\notin I, I\cup \{i\} \notin \mathrm{Star}(\sigma(\gamma)).
\end{split}
\end{equation*}
\end{definition}

Now we describe these definitions in our situation of $\mathrm{rk} \,\mathrm{N}=2$. The set $\mathrm{Box}(\Sigma)$ consists of $(0,0)$ and all $\gamma$ of degree one in the interior of cones of $\Sigma$, see Figure \ref{sigma2}.

\begin{figure}[H]
  \includegraphics[width=0.46\textwidth]{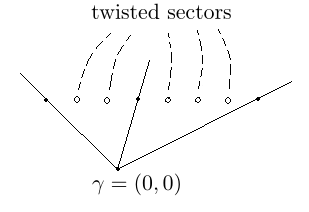}
  \caption{}
  \label{sigma2}
\end{figure}

\begin{proposition}\label{newH}
For $\gamma\neq (0,0)$, we have $\mathrm{dim} \, H_{\gamma}=1$. We denote the generator of $H_{\gamma}$ by $1_{\gamma}$ so that $H_{\gamma}=\mathbb{C}1_{\gamma}$.
For $\gamma=(0,0)$, we have $\mathrm{dim}(H_{\gamma})=\#\Sigma_{(1)}-1$. The cohomology $H_{\gamma}$ is the quotient of $\mathbb{C}[D_i,i\in \Sigma]$ by the ideal generated by
\begin{equation*}
\begin{split}
&D_iD_j, \text{ for any }i,j \in \Sigma, \ \ \ D_i \text{ for } i\notin \Sigma,\\
&\sum_{i\in \Sigma}D_i, \ \ \  \sum_{i\in \Sigma}iD_i.
\end{split}
\end{equation*}
\end{proposition}

\begin{proof}
This proposition follows from Definition \ref{oldH}.

\smallskip
When $\gamma\neq (0,0)$, we have $\mathrm{Star}(\sigma(\gamma))=\sigma(\gamma)$. Thus $H_{\gamma}$ is isomorphic to $\mathbb{C}$, which implies our result.
When $\gamma= (0,0)$, we have $\sigma(\gamma)=\emptyset$. Thus $\mathrm{Ann}(v_i, i\in \sigma(0,0))=N^{\vee}$. Thus we have two relations $\sum_{i\in \Sigma}D_i=0$ and $\sum_{i\in \Sigma}iD_i=0$.

\smallskip
By Definition \ref{oldH}, we have $D_iD_j=0$ if $v_i, v_j \in \Sigma$ are not adjacent. Let $\Sigma_{(1)}=\{v_{i_0},v_{i_1},\ldots,v_{i_{r+1}}\}$, where $0=i_0<i_1<\ldots<i_{r+1}=n$. Now we claim $D_{i_0}D_{i_1}=0$. With the two relations, we can express $D_{i_1}$ as linear combination of $D_{i_k}, k\geq2$. Then $v_{i_k}, k\geq2$ are not adjacent to $v_{i_0}$. Thus we get $D_{i_0}D_{i_1}=0$. Next we claim $D_{i_1}D_{i_2}=0$. With the two relations, we can express $D_{i_2}$ as linear combination of $D_{i_0}$ and $D_{i_k}, k\geq 3$. Since $v_{i_k}, k\geq 3$ are not adjacent to $v_{i_1}$ and $D_{i_0}D_{i_1}=0$, we get $D_{i_1}D_{i_2}=0$. Similarly, we get $D_{i_k}D_{i_{k+1}}=0$ for $k=0,\ldots,r$. Then by expressing $D_{i_k}$ as linear combination of $D_{i_l}, l\neq k$, we get $(D_{i_k})^2=0$ for $k=0,\ldots,r+1$.
\end{proof}

\begin{corollary}\label{basisH}
Let $\Sigma^{\circ}$ be the interior of $\Sigma$.
The basis of $H_{(0,0)}$ can be given by $1_{(0,0)}$ and $D_i, i\in \Sigma^{\circ}$. The basis of $H$ can be given by $1_{\gamma}$ for all twisted sectors $\gamma\neq(0,0)$, $1_{(0,0)}$ and $D_i, i\in \Sigma^{\circ}$.
\end{corollary}

\begin{proposition}\label{newH^c}
For $\gamma \neq (0,0)$, we have that $H^c_{\gamma}$ is over $\mathbb{C}$ generated by $F_{I}$, where $I=\sigma(\gamma)$. For simplicity, we write $F_I$ as $F_{\emptyset, \gamma}$ and $H^c_{\gamma}=\mathbb{C}F_{\emptyset,\gamma}$.
For $\gamma = (0,0)$, we have that $H^c_{\gamma}$ is a module over $H_{(0,0)}$ generated by $F_{\{i,j\}}$ and $F_k$ for two dimensional cones $\{i,j\} \in\Sigma$ and $k\in \Sigma$ such that $v_k\in \Sigma^{\circ}$ with relations
\begin{equation*}
\begin{split}
&D_iF_j-F_{\{i,j\}} \text{ for } i\neq j, \text{ if }\{i,j\} \text{ is a two dimensional cone in } \Sigma,\\
&D_iF_j \text{ for } i\neq j, \text{ if }i,j \in \Sigma \text{ are not adjacent},\ \ \ D_kF_{\{i,j\}} \text{ for any } k\in \Sigma .
\end{split}
\end{equation*}
\end{proposition}
\begin{proof}
By Definition \ref{oldH^c}, we have $D_kF_{\{i,j\}}=0$ for $k\neq i,j$ in $H^c_{(0,0)}$. For $k\in \{i,j\}$, we express $D_k$ as linear combination of $D_l, l\neq i,j$. Then we also have $D_kF_{\{i,j\}}=0$ in the case of $k\in \{i,j\}$.
Other results in this proposition also follow from Definition \ref{oldH^c} and the  argument similar to the one in the proof of Proposition \ref{newH} which left to the reader.
\end{proof}

\begin{proposition}\label{F_ij}
Let $\Sigma_{(1)}=\{v_{i_0},v_{i_1},\ldots,v_{i_{r+1}}\}$, where $0=i_0<i_1<\ldots<i_{r+1}=n$.
There exists a unique element $F$ in $H^c_{(0,0)}$ such that $F_{\{i,j\}}=\frac{1}{|j-i|}F$ for all adjacent $i,j \in\Sigma$. We also have $D_{i_k}F_{i_k}=-(\frac{1}{i_k-i_{k-1}}+\frac{1}{i_{k+1}-i_k})F$.
\end{proposition}
\begin{proof}
Let $i_{k-1}<i_k< i_{k+1}$ be three adjacent elements of $\Sigma_{(1)}$.
Consider a relation on $D_i$
$$
\sum_{i} (i-i_{k}) D_i = 0.
$$
After multiplying by $F_{i_k}$ and using  the vanishing of all but two $D_i F_{i_k}$, we get $(i_{k} -i_{k-1})F_{i_{k-1},i_k} =(i_{k+1} -i_k)F_{i_{k},i_{k+1}}$ and the first statement follows. The second statement follows from $D_{i_k}F_{i_k}=-\sum_{j\neq k}D_{i_j}F_{i_k}=-F_{i_{k-1},i_k}-F_{i_k,i_{k+1}}.$
\end{proof}

\begin{corollary}\label{basisH^c}
The basis of $H^c_{(0,0)}$ are $F$ and $F_i, i\in \Sigma^{\circ}$. The basis of $H^c$ are $F_{\emptyset,\gamma}$ for all twisted sector $\gamma\neq(0,0)$, $F$ and $F_i, i\in \Sigma^{\circ}$.
\end{corollary}

\begin{proof}
The degree one in $D$, $F$ subspace of  $H^c_{(0,0)}$ is freely generated by $F_i$ for $i\in \Sigma^\circ$. The above proposition shows that the degree two subspace is one-dimensional and generated by $F$.
\end{proof}

In \cite{BorisovHorja1,BorisovHorja}, the authors calculate $K_0(\mathbb{P}_{\Sigma})$ and the module $K_0^c(\mathbb{P}_{\Sigma})$ over $K_0(\mathbb{P}_{\Sigma})$.

\begin{definition}\label{defineK0}
\cite{BorisovHorja1}
The $K-$theory $K_0(\mathbb{P}_{\Sigma})$ is the quotient of the ring $\mathbb{C}[R_i,R_i^{-1}], 0\leq i \leq n$ by the relations
\begin{equation*}
\prod_{i=0}^{n}R_i^{m\cdot v_i}-1, m\in N^{\vee} \ \ and  \ \
\prod_{i\in I}(1-R_i), I \notin \Sigma.
\end{equation*}
\end{definition}

We have the following statement in the case of $\mathrm{rk\, N}=2$.

\begin{proposition}
The $K-$theory $K_0(\mathbb{P}_{\Sigma})$ is the quotient of the ring $\mathbb{C}[R_i,R_i^{-1}],$     $ 0\leq i \leq n$ by the relations
\begin{equation*}
\begin{split}
&R_i-1 \text{ for } i\notin \Sigma,\quad \prod_{i\in \Sigma}R_i -1,\quad \prod_{i\in \Sigma}R^i_i -1,\\
&(R_i-1)(R_j-1) \text{ for all pairs } \{i,j\} \text{ such that } i<j, i \text{ and } j  \text{ are not adjacent in } \Sigma.\\
\end{split}
\end{equation*}
\end{proposition}

\begin{proof}
Follows from Definition \ref{defineK0}.
\end{proof}

\begin{definition}\label{defineK0^c}
\cite{BorisovHorja}
The module $K_0^c(\mathbb{P}_{\Sigma})$ is defined to be a module over $K_0(\mathbb{P}_{\Sigma})$ generated by $G_I$, where $I\in \Sigma$, $\sigma_{I}^{\circ}\subseteq C^{\circ}$, with relations
\begin{equation*}
\begin{split}
&(1-R_i^{-1})G_I=G_{I\cup \{i\}}, \text{ if }I\cup \{i\} \in \Sigma,\\
&(1-R_i^{-1})G_I=0, \text{ if }I\cup \{i\}\notin \Sigma,
\end{split}
\end{equation*}
for all $i$, $I$ such that $i\notin I$.
\end{definition}

We have the following statement in the case of $\mathrm{rk\, N}=2$.

\begin{proposition}
The module $K_0^c(\mathbb{P}_{\Sigma})$ is a module over $K_0(\mathbb{P}_{\Sigma})$ generated by $G_I$, where $I$ is either a ray in the interior of $\Sigma$ or a two dimensional cone in $\Sigma$, with relations
\begin{equation*}
\begin{split}
&(1-R_i^{-1})G_j=G_{\{i,j\}}, \text{ if }\{i,j\} \text{ is a two dimensional cone in } \Sigma,\\
&(1-R_i^{-1})G_j=0, \text{ if }i,j \in \Sigma \text{ are not adjacent},\\
&(1-R_k^{-1})G_{\{i,j\}}=0 \text{ for } k\neq i,j.
\end{split}
\end{equation*}
\end{proposition}
\begin{proof}
Follows from Definition \ref{defineK0^c}.
\end{proof}

The following proposition of \cite{BorisovHorja}  describes the analogs of Chern class isomorphisms in this situation.
\begin{proposition}\label{definech,ch^c}
\cite{BorisovHorja0,BorisovHorja}
	There are natural algebra isomorphisms
	\begin{equation*}
	ch:K_0(\mathbb{P}_{\Sigma})\xrightarrow{\sim}\bigoplus_{\gamma\in\mathrm{Box}(\Sigma)}H_{\gamma},\quad ch^c:K_0^c(\mathbb{P}_{\Sigma})\xrightarrow{\sim}\bigoplus_{\gamma\in\mathrm{Box}(\Sigma)}H_{\gamma}^c.
	\end{equation*}
The map $ch$ is defined to be $\oplus_{\gamma}ch_{\gamma}$. The map $ch_{\gamma}:K_0(\mathbb{P}_{\Sigma}) \rightarrow H_{\gamma}$ is given by
\begin{equation}\label{ch_gamma}
\begin{split}
&ch_{\gamma}(R_i)=1,    \ \ \ \ \ \ \ \ \ \ \ \ \ \ \ \ \ i\notin \mathrm{Star}(\sigma(\gamma))\\
&ch_{\gamma}(R_i)=e^{\overline{D}_i},    \ \ \ \ \ \  \ \ \ \ \ \ \ \ \  i\in \mathrm{Star}(\sigma(\gamma))-\sigma(\gamma)\\
&ch_\gamma({R_i})=e^{2\pi {\rm i} \gamma_i}\prod_{j\notin \sigma(\gamma)}ch_{\gamma}(R_j)^{m_i \cdot v_j}, \ \ i\in \sigma(\gamma)
\end{split}
\end{equation}
where $m_i$ is any $\mathbb{Q}-$valued linear function on $N$ which takes values $-1$ on $v_i$ and $0$ on all other $v_j$, $j\in \sigma(\gamma)$.

The $ch^c$ is defined to be $\oplus_{\gamma}ch^c_{\gamma}$. The projection $ch^c_{\gamma}:K^c_0(\mathbb{P}_{\Sigma}) \rightarrow H^c_{\gamma}$ is given by
$ch^c_{\gamma}(\prod_{i=0}^n R_i^{l_i}G_I)=0$ for $I\nsubseteq \mathrm{Star}(\sigma(\gamma))$ and
\begin{equation}\label{ch^c_gamma}
\begin{split}
ch^c_{\gamma}(\prod_{i=0}^n R_i^{l_i}G_I)=&\prod_{i=0}^n ch_{\gamma}(R_i)^{l_i}\prod_{i\in I, i\notin \sigma(\gamma)}(\frac{1-e^{-\overline{D}_i}}{\overline{D}_i})\\
&\cdot\prod_{i\in I\cap \sigma(\gamma)}(1-ch_{\gamma}(R_i)^{-1})\overline{F}_{\overline{I},\gamma}
\end{split}
\end{equation}
for $I\in\mathrm{Star}(\sigma(\gamma))$. The cone $\overline{I}$ in the induced fan is defined by the set of indices in $I$, but not in $\sigma(\gamma)$. The
$\overline{F}_{\overline{I},\gamma}$ indicates the generator of $H_{\gamma}^{c}$ that corresponds to $\overline{I}$ in the induced fan $\Sigma_{\gamma}$ \cite{BorisovHorja}.
\end{proposition}

We now introduce a duality involution on $H$ which is meant to be a cohomology analog of the duality involution in $K$-theory.
\begin{definition}
Let $\gamma=\sum_{k\in \sigma(\gamma)}\gamma_kv_k$ with $0 <\gamma_k<1$ be a twisted sector of $\Sigma$.  The dual of $\gamma$ is defined to be the twisted sector
$$\gamma^{*}=\sum_{k\in \sigma(\gamma)}(1-\gamma_k)v_k=\sum_{k\in \sigma(\gamma)}v_k- \gamma.$$
We define a duality map $*: H \rightarrow H$ by $(1_{\gamma})^{*}=1_{\gamma^{*}}$ and $(D_i)^{*}=-D_i$. In particular, $*$ maps $H_{\gamma}$ to $H_{\gamma^{*}}$.
\end{definition}

\begin{lemma}
For any element $a=\prod_{i=0}^{n}R_i^{l_i}$, we have $ch_{\gamma}(a^{*})=(ch_{\gamma^{*}}(a))^{*}$. Then $ch(a^{*})=(ch(a))^{*}$.
\end{lemma}
\begin{proof}
By equation \eqref{ch_gamma}, we have
$$ch(a^{*})=\prod_{i\in \mathrm{Star}(\sigma(\gamma))-\sigma(\gamma)} e^{-\overline{D}_il_i}\prod_{i\in\sigma(\gamma)}\big(e^{2\pi {\rm i} \gamma_i(-l_i)}\prod_{j\notin \sigma(\gamma)}ch_{\gamma}(R_j)^{m_i \cdot v_j}\big),$$
 $$ch_{\gamma^{*}}(a)=\prod_{i\in \mathrm{Star}(\sigma(\gamma^{*}))-\sigma(\gamma^{*})} e^{\overline{D}_il_i}\prod_{i\in\sigma(\gamma^{*})}\big(e^{2\pi {\rm i} (1-\gamma_i)(l_i)}\prod_{j\notin \sigma(\gamma^{*})}ch_{\gamma}(R_j)^{m_i \cdot v_j}\big).$$
 Since $(D_i)^{*}=-D_i$ and $\sigma(\gamma^{*})=\sigma(\gamma)$, we get our result.
\end{proof}

\begin{definition}\label{chiH}
We define the map $\chi_{H}: H^{c}\rightarrow \mathbb{ C}$ to be
$$\chi_{H}(\beta)=\sum_{\gamma}\frac{1}{|\mathrm{Box}(\sigma(\gamma))|}\int_{\gamma}Td(\gamma)\beta,$$
where $\beta\in H^{c}$ and $Td(\gamma):=\prod_{i\in\sigma(\gamma)}\frac{1}{1-ch(R_i^{-1})}\prod_{\substack{i\in \mathrm{Star}(\sigma(\gamma))\\ i\notin\sigma(\gamma)}}(\frac{D_i}{1-ch(R_i^{-1})})$.
We define the Euler characteristic pairing
$$\chi_{H}:H \times H^c\rightarrow \mathbb{Z}$$
by $$\chi_{H}(\alpha,\beta)=\chi_{H}(\alpha^{*}\beta),$$
for any $\alpha\in H$ and $\beta\in H^c$.
We regard $\chi_{H}$ as an element of $(H\otimes H^c)^{\vee}$.
\end{definition}

\begin{definition}
We define the dual map $*:K_0(\mathbb{P}_{\Sigma}) \rightarrow K_0(\mathbb{P}_{\Sigma})$ by $(R_i)^{*}=R_i^{-1}$.
\end{definition}

Our definition of $\chi_{H}$ was chosen specifically to be compatible with the Euler characteristic pairing $\chi$ between $K_0(\mathbb{P}_{\Sigma})$ and $K^c_0(\mathbb{P}_{\Sigma})$ defined in \cite{BorisovHorja}.
\begin{lemma}\label{chvee}
For any $a\in K_0(\mathbb{P}_{\Sigma})$ and $b\in K^c_0(\mathbb{P}_{\Sigma})$, we have $$\chi_{H}(ch(a),ch^{c}(b))=\chi(a,b).$$
\end{lemma}
\begin{proof}
By definition of $\chi$ and Proposition $4.5$ of \cite{BorisovHorja}, we get
\begin{equation*}
\begin{split}
\chi(a,b)=\chi(a^{*}b)=&\sum_{\gamma}\frac{1}{|\mathrm{Box}(\sigma(\gamma))|}\int_{\gamma}Td(\gamma)ch_{\gamma}^c(a^{*}b).
\end{split}
\end{equation*}
Propsition $3.11$ in \cite{BorisovHorja} shows that $ch^c$ is compatible with $ch$, so we have $ch_{\gamma}^c(a^{*}b)=ch_{\gamma}(a^{*})ch^c_{\gamma}(b)$.
Then by Lemma \ref{chvee}, we get $ch_{\gamma}(a^{*})ch^c_{\gamma}(b)=ch_{\gamma^{*}}(a)^{*}ch^c_{\gamma}(b)$. Therefore we get
\begin{equation*}
\begin{split}
\chi(a,b)=&\sum_{\gamma}\frac{1}{|\mathrm{Box}(\sigma(\gamma))|}\int_{\gamma}Td(\gamma)ch_{\gamma^{*}}(a)^{*}ch^c_{\gamma}(b)\\
=&\chi_{H}(ch(a)^{*}ch^c(b))=\chi_{H}(ch(a),ch^{c}(b)).
\end{split}
\end{equation*}
\end{proof}
We now calculate $\chi_{H}$ for $\mathrm{rk \, N}=2$.
We first compute the part of $\chi_{H,\gamma}$ contributed by all twisted sectors $\gamma\neq(0,0)$ which we denote by $(\chi_{H})_{\text{twisted}}$.
\begin{lemma}\label{chiHtw}
We have
$$(\chi_{H})_{\text{twisted}}=\sum_{\substack{\gamma\text{ is a twisted sector }\\ \sigma(\gamma)=\sigma_{ij}}}\frac{1}{(j-i)4 \, \mathrm{sin}^2(\pi{\rm i} \gamma_i)} \,1^{\vee}_{\gamma}\otimes F_{\emptyset, \gamma^{*}}^{\vee}.$$
\end{lemma}
\begin{proof}
For  $\gamma\neq(0,0)$
and $\sigma(\gamma)=\sigma_{ij}$, we have
$$Td(\gamma)=(\frac{1}{1-e^{-2\pi {\rm i} \gamma_i}})(\frac{1}{1-e^{-2\pi {\rm i} \gamma_j}})=\frac{1}{4\, \mathrm{sin}^2(\pi{\rm i} \gamma_i)}.$$
Therefore, by Definition \ref{chiH},
$$
\chi_{H,\gamma^*} (1_{\gamma}, F_{\emptyset,\gamma^*}) = \frac 1{(j-i)}\int_\gamma \frac{1}{4\, \mathrm{sin}^2(\pi{\rm i} \gamma_i)} 1_{\gamma^*} F_{\emptyset,\gamma^*} = \frac{1}{(j-i)4\, \mathrm{sin}^2(\pi{\rm i} \gamma_i)}
$$
and the result follows.
\end{proof}

\begin{lemma}\label{chiH00}
We have $$(\chi_{H})_{(0,0)}=1_{(0,0)}^{\vee}\otimes F^{\vee}- \sum_{0<k,l<n}m_{kl}D_{i_k}^{\vee}\otimes F_{i_l}^{\vee},$$
where \begin{equation*}
	m_{kl}=\left\{
	\begin{array}{l}
	0,\text{ if }k\not=l,l-1,l+1\\
	\frac{1}{|i_l-i_k|},\text{ if }|l-k|=1\\
	-\frac{1}{i_l-i_{l-1}}-\frac{1}{i_{l+1}-i_{l}},\text{ if }l=k.\\
	\end{array}
	\right.
	\end{equation*}
\end{lemma}
\begin{proof}
When $\gamma=(0,0)$, we have $$Td(\gamma)=\prod_{i\in \Sigma}\frac{D_i}{1-ch(R_i^{-1})}=\prod_{i\in \Sigma}\frac{D_i}{1-e^{-D_i}}=\prod_{i\in \Sigma}(1+\frac{D_i}{2})=1+\frac{1}{2}\sum_{i\in \Sigma}D_i=1.$$

Let $\alpha=\sum_{i\in \Sigma^{\circ}}s_iD_i+s1_{(0,0)}\in H_{(0,0)}$ and $\beta=\sum_{j\in \Sigma^{\circ}}t_jF_j+tF\in H^c_{(0,0)}$.
Let $\Sigma_{(1)}=\{i_0=0,i_1,\cdots,i_r,i_{r+1}=n\}$, where $i_0<\ldots<i_{r+1}=n$. We get
\begin{equation}
\begin{split}
\chi_{H,(0,0)}(\alpha,\beta)=&\chi_{H,(0,0)}\big((s1_{(0,0)}-\sum_{i\in \Sigma^{\circ}}s_iD_i)(tF+\sum_{j\in \Sigma^{\circ}}t_jF_j)\big)\\
=&st-\sum_{i,j \text{ are adjacent in } \Sigma^{\circ}}s_it_j\int_{\gamma} D_iF_j-\sum_{i\in \Sigma^{\circ}}s_it_i\int_{\gamma}D_iF_i\\
=&st-\sum_{k,l |k-l|=1 }s_{i_k}t_{i_l}\int_{\gamma}D_{i_k}F_{i_l}-\sum_{k}s_{i_k}t_{i_k}\int_{\gamma}D_{i_k}F_{i_k}.
\end{split}
\end{equation}

Thus, in the basis $1_{(0,0)}^{\vee}, D_{i_1}^{\vee},\ldots, D_{i_r}^{\vee}$ of $H^{\vee}$ and basis $F^{\vee}, F^{\vee}_{i_1},\ldots,F^{\vee}_{i_r}$ of $H^{c,\vee}$, we can express $\chi_{H,(0,0)}$ as the matrix
\begin{equation*}
	\begin{pmatrix}
	1 & 0  \\
	0 & -M
\end{pmatrix},
	\end{equation*}
where the entries in $M$ are computed by Proposition \ref{F_ij}
\begin{equation*}
	m_{kl}=\int_{\gamma}D_{i_k}F_{i_l}=\left\{
	\begin{array}{l}
	0,\text{ if }k\not=l,l-1,l+1\\
	\frac{1}{|i_l-i_k|},\text{ if }|l-k|=1\\
	-\frac{1}{i_l-i_{l-1}}-\frac{1}{i_{l+1}-i_{l}},\text{ if }l=k.\\
	\end{array}
	\right.
	\end{equation*}
We have
\begin{equation*}
	M=\begin{pmatrix}
	-\frac{1}{i_2-i_1}-\frac{1}{i_1-i_0} & \frac{1}{i_2-i_1} & 0 & \ldots \\
	\frac{1}{i_2-i_1} & -\frac{1}{i_3-i_2}-\frac{1}{i_2-i_1} & \frac{1}{i_3-i_2} & \ldots \\
	0 & \frac{1}{i_3-i_2} & -\frac{1}{i_4-i_3}-\frac{1}{i_3-i_2} & \ldots \\
	\vdots & \vdots & \vdots & \ddots
	\end{pmatrix}.
	\end{equation*}
\end{proof}

In what follows, we will need the inverse of $M$, which has a surprisingly simple form.
\begin{lemma}\label{inversM}
The inverse of $M$ is $G=(g_{kl})$, where $g_{kl}=-\frac{i_k}{n}(n-i_l)$ if $k\leq l$ and $g_{kl}=-\frac{i_l}{n}(n-i_k)$ if $k\geq l$.
\end{lemma}
\begin{proof}
The $k-$th row of $M$ is $$(0,\ldots,0, \frac{1}{i_k-i_{k-1}} , -\frac{1}{i_{k+1}-i_k}-\frac{1}{i_k-i_{k-1}} , \frac{1}{i_{k+1}-i_k},0,\ldots,0 )$$ and the $l-$th column of $G$ is $$-(*,\frac{i_{l-1}}{n}(n-i_l),\frac{i_{l}}{n}(n-i_l),\frac{i_{l}}{n}(n-i_{l+1}),*)^{T}.$$
If $l\geq k+1$, the $k-$th row of $M$ multiplied with the $l-$th column of $G$ equals
$$
\hskip -10pt\begin{array}{l}
{-\frac{1}{i_k-i_{k-1}}\frac{i_{k-1}}{n}(n-i_l)-(-\frac{1}{i_{k+1}-i_k}-\frac{1}{i_k-i_{k-1}})
\frac{i_{k}}{n}(n-i_l)-\frac{1}{i_{k+1}-i_k}\frac{i_{k+1}}{n}(n-i_l)=0.}
\end{array}
$$
If $l\leq k-1$, the $k-$th row of $M$ multiplied with the $l-$th column of $G$ equals
$$
\hskip -10pt\begin{array}{l}
{-\frac{1}{i_k-i_{k-1}}\frac{i_{l}}{n}(n-i_{k-1})-(-\frac{1}{i_{k+1}-i_k}-\frac{1}{i_k-i_{k-1}})
\frac{i_{l}}{n}(n-i_k)-\frac{1}{i_{k+1}-i_k}\frac{i_{l}}{n}(n-i_{k+1})=0.}
\end{array}
$$
If $l= k$, the $k-$th row of $M$ multiplied with the $l-$th column of $G$ equals
$$
\hskip -10pt\begin{array}{l}
{-\frac{1}{i_k-i_{k-1}}\frac{i_{k-1}}{n}(n-i_{k})-(-\frac{1}{i_{k+1}-i_k}-\frac{1}{i_k-i_{k-1}})
\frac{i_{k}}{n}(n-i_k)-\frac{1}{i_{k+1}-i_k}\frac{i_{k}}{n}(n-i_{k+1})=1.}
\end{array}
$$
\end{proof}
The formula for $\chi_{H}$ leads to the following formula for the inverse $\chi_{H}^{-1} \in H\otimes H^c$.
\begin{theorem}\label{chiH-1}
We have
\begin{equation}
\begin{split}
\chi_{H}^{-1}=&\sum_{\gamma\neq(0,0),\sigma(\gamma)=\sigma_{ij}}(j-i)4\,\mathrm{sin}^{2}\pi {\rm i} \gamma_i 1_{\gamma}\otimes F_{\emptyset, \gamma^{*}}\\
+&1_{(0,0)}\otimes F-\sum_{0<k,l<n}g_{kl}D_{i_k}\otimes F_{i_l},
\end{split}
\end{equation}
where $g_{kl}=-\frac{i_k}{n}(n-i_l)$ if $k\leq l$ and $g_{kl}=-\frac{i_l}{n}(n-i_k)$ if $k\geq l$.
\end{theorem}
\begin{proof}
Follows from Lemma \ref{chiHtw}, Lemma \ref{chiH00} and Lemma \ref{inversM}.
\end{proof}

\section{Proof of duality conjectures}\label{pluginGamma}
In this section, we consider Gamma series isomorphism between the solutions to the better-behaved GKZ systems and the
$K$-theory spaces  given in \cite{BorisovHorja}. In the case of $\mathrm{rk \,N}=2$, we prove Conjectures \ref{7.3} and \ref{7.1} which are the duality conjectures formulated in \cite{BorisovHorja}.

\smallskip
The key ingredients of Conjecture \ref{7.3} are the Gamma series $\Gamma$ and $\Gamma^{\circ}$ which give solutions to $bbGKZ(C,0)$ and $bbGKZ(C^{\circ},0)$ with values in $\oplus_{\gamma}H_{\gamma}$ and $\oplus_{\gamma}H_{\gamma}^c$ respectively. The series converge absolutely and uniformly on compacts in some region of $(x_i)$, see \cite{BorisovHorja,BorisovHorja2}.
\begin{definition}\label{gammac}
\cite{BorisovHorja}
	Consider for each $c\in C$ and each twisted sector $\gamma=\sum_i \gamma_i v_i$ the set $L_{c,\gamma}$ of $(l_i)\in\mathbb{Q}^n$ such that $\sum_i l_i v_i=-c$ and $l_i-\gamma_i\in\mathbb{Z}$. Define
	\begin{equation*}
	(\Gamma(x_1,\cdots,x_n))_c=\bigoplus_{\gamma}\sum_{(l_i)\in L_{c,\gamma}}\prod_{i=1}^n\frac{x_i^{l_i+\frac{D_i}{2\pi i}}}{\Gamma(1+l_i+\frac{D_i}{2\pi i})}.
	\end{equation*}
	Moreover, if $c\in C^{\circ}$, define the set $\sigma=\{i|l_i<0\}$. Define
	\begin{equation*}
	(\Gamma^{\circ}(x_1,\cdots,x_n))_c=\bigoplus_{\gamma}\sum_{(l_i)\in L_{c,\gamma}}\prod_{i=1}^n\frac{x_i^{l_i+\frac{D_i}{2\pi i}}}{\Gamma(1+l_i+\frac{D_i}{2\pi i})}(\prod_{i\in\sigma}D_i^{-1})F_{\sigma}
	\end{equation*}
	with $D_i=\log ch_{\gamma}(R_ie^{-2\pi i\gamma_i})$.
\end{definition}

\begin{proposition}\label{gammasolution}
\cite{BorisovHorja}
Let $bbGKZ(C,0)$ and $bbGKZ(C^{\circ},0)$ be the spaces of solutions to the corresponding better behaved hypergeometric systems. Then the Gamma series functions $\Gamma: K_0(\mathbb{P}_{\Sigma})^{\vee} \rightarrow bbGKZ(C,0)$ and $\Gamma^{\circ}: K^c_0(\mathbb{P}_{\Sigma})^{\vee} \rightarrow bbGKZ(C^{\circ},0)$ are isomorphisms of linear spaces.
\end{proposition}

From Definition \ref{gammac}, we know $\Gamma_c\in\oplus_{\gamma}H_{\gamma}$ and $\Gamma^{\circ}_d\in\oplus_{\gamma}H^c_{\gamma}$ for $c\in C$ and $d\in C^{\circ}$. After substituting them into the pairing in Definition \ref{pairing} in Section \ref{pairingofsol}, we get $\langle\Gamma,\Gamma^{\circ}\rangle\in(\oplus_{\gamma}H_{\gamma})\otimes(\oplus_{\gamma}H^c_{\gamma})$.
We also consider the Euler characteristic pairing
\begin{equation*}
\chi:K_0(\mathbb{P}_{\Sigma})\times K_0^c(\mathbb{P}_{\Sigma})\rightarrow\mathbb{C}
\end{equation*}
which is defined in \cite{BorisovHorja}.

Now we will prove the last claim in Conjecture \ref{7.3} (Conjecture 7.3 in \cite{BorisovHorja} )in the case of $\mathrm{rk \, N}=2$, by an explicit calculation.

\smallskip

{\bf Notations.} We write the formulas in Definition \ref{gammac} as follows
\begin{equation*}
\Gamma_c=\bigoplus_{\gamma}\Gamma_c^{\gamma}=\bigoplus_{\gamma}\sum_{(l_k)\in L_{c,\gamma}}\Gamma_c^{\gamma,(l_k)},\quad \Gamma_d^{\circ}=\bigoplus_{\gamma}\Gamma_d^{\circ,\gamma}=\bigoplus_{\gamma}\sum_{(r_k)\in L_{d,\gamma}}\Gamma_d^{\circ,\gamma,(r_k)},
\end{equation*}
where $\Gamma_c^{\gamma,(l_k)}=\prod_{k=1}^n\frac{x_k^{l_k+\frac{D_k}{2\pi {\rm i}}}}{\Gamma(1+l_k+\frac{D_k}{2\pi {\rm i}})}$ and  $\Gamma_d^{\circ,\gamma,(r_k)}=\prod_{k=1}^n\frac{x_k^{l_k+\frac{D_k}{2\pi {\rm i}}}}{\Gamma(1+l_k+\frac{D_k}{2\pi {\rm i}})}(\prod_{k\in\sigma}D_k^{-1})F_{\sigma}.$
Then the terms in Definition \ref{pairing} without $\Gamma_{(0,0)}$ can be written as

\begin{equation}\label{twistno00}
-n\sum_{0\leq i<j\leq n}(j-i)x_ix_j\sum_{\substack{c,d \in \sigma_{ij} \\c+d=v_i+v_j}}\delta_{ij}^{c}\bigoplus_{\gamma,\gamma'}\sum_{\substack{(l_k)\in L_{c,\gamma} \\ (r_k)\in L_{d,\gamma'}}}\Gamma_c^{\gamma,(l_k)}\Gamma_d^{\circ,\gamma',(r_k)}.
\end{equation}

\smallskip

{\bf Main idea.} By Theorem \ref{pairingthm}, we know the pairing $\langle\Gamma,\Gamma^{\circ}\rangle$ in Definition \ref{pairing} is a constant. Therefore, in order to compute $\langle\Gamma,\Gamma^{\circ}\rangle$, we only need to calculate the constant contribution of each terms. Although there are many terms in $\langle\Gamma,\Gamma^{\circ}\rangle$, only few of them have nonzero contribution to the constant.

\medskip

First, we consider the terms of the twisted sectors, see Lemma \ref{dualtwist} and Proposition \ref{constwist}.
Second, we calculate the contribution given by the terms of the untwisted sector $\gamma=(0,0)$, see Lemma \ref{conditionuntwisted}, Lemma \ref{(2)}, Lemma \ref{(3)} and Proposition \ref{conuntwist}. Then we add these contributions to the constant together to get our result Theorem \ref{pairingforgamma}.

\medskip

{\bf Step 1.} First, we consider the terms of twisted sectors.
 Let $\gamma$ and $\gamma'$ be two twisted sectors. We consider the term $x_ix_j\Gamma_c^{\gamma,(l_k)}\Gamma_d^{\circ,\gamma',(r_k)}$ in the summation and have the following lemma.
\begin{lemma}\label{dualtwist}
For $x_ix_j\Gamma_c^{\gamma,(l_k)}\Gamma_d^{\circ,\gamma',(r_k)}$ to contribute a nonzero constant term, it should satisfy the following conditions:
\begin{enumerate}
\item $\gamma$ and $\gamma'$ are dual to each other, i.e.,$\gamma'=\gamma^{*}$;
\item $r_k=-l_k=0$ if $k\not=i,j$ and $r_k=-1-l_k$ if $k=i,j$;
\item The cone $\sigma_{ij}$ generated by $v_i$ and $v_j$ is the minimal cone in $\Sigma$ which contains the twisted sector $\gamma$;
\item $c=\gamma^{*}$ and $d=\gamma$.
\end{enumerate}
\end{lemma}

\begin{proof}
 By the definition of $L_{c,\gamma}$ and $L_{d,\gamma^{\prime}}$, we have $l_k\equiv\gamma_k\mod\mathbb{Z}$ and $r_k\equiv\gamma_k^{\prime}\mod\mathbb{Z}$ for all $k=0,\ldots,n$. Therefore, we get $l_k+r_k\equiv\gamma_k+\gamma_k^{\prime}\mod\mathbb{Z}$. A necessary condition for $x_ix_j\Gamma_c^{\gamma,(l_k)}\Gamma_d^{\circ,\gamma',(r_k)}$ to contribute nonzero constant is that the degrees of $x_k$ are integers for all $k$. Thus  $l_k+r_k$ are integers, which implies that $\gamma_k+\gamma_k^{\prime}$ are integers for all $k=0,\ldots,n$. Now by the condition that $\gamma,\gamma^{\prime}\in(0,1)$, we get $\gamma+\gamma^{\prime}=1$  which implies the condition $(1)$.

\smallskip
 Moreover, one easily sees that degree of $x_k$ in $x_ix_j\Gamma_c^{\gamma,(l_k)}\Gamma_d^{\circ,\gamma',(r_k)}$ equals $l_k+r_k$ if $k\not=i,j$ and $l_k+r_k+1$ if $k=i,j$. In order to get a nonzero constant, all these degrees must be $0$, which gives
	\begin{equation}\label{*}
	r_k=-l_k\text{ if }k\not=i,j,\quad r_k=-1-l_k\text{ if }k=i,j.
	\end{equation}
Note that after expanding $x_ix_j\Gamma_c^{\gamma,(l_i)}\Gamma_d^{\circ,\gamma',(r_i)}$, every term in it is a multiple of $D_{\{k|l_k\in\mathbb{Z}_{<0}\}}\, 1_{\gamma}\otimes F_{\{k|r_k\in\mathbb{Z}_{<0}\}}\, 1_{\gamma^{*}}$. By relations in $H_{\gamma}$ and $H_{\gamma}^c$, we know that the term is nonzero only when both $\{k|l_k\in\mathbb{Z}_{<0}\}$ and $\{k|r_k\in\mathbb{Z}_{<0}\}$ are empty, i.e., all $l_k$ and $r_k$ are non-negative. Together with Equation \eqref{*}, we get $l_k=r_k=0$ if $k\not=i,j$, which implies condition $(2)$.

\smallskip
Let the cone $\sigma_{pq}$ with rays $v_p$ and $v_q$ be the minimal cone of $\Sigma$ that contains $\gamma$ for $p<q$ and $p,q\in\{0,\ldots,n\}$. We write $\gamma=\gamma_pv_p+\gamma_qv_q$, where $0<\gamma_p,\gamma_q<1$. If $i\notin\{p,q\}$, we have $\gamma_i=0$. Since $l_i\equiv\gamma_i\mod\mathbb{Z}$, we get $l_i\in \mathbb{Z}$. Similarly, we have $r_i\in \mathbb{Z}$. Since $r_i=-1-l_i$, then either $r_i<0$ or $l_i<0$. Thus we have either $\{k|l_k\in\mathbb{Z}_{<0}\}\neq\emptyset$ or $\{k|r_k\in\mathbb{Z}_{<0}\}\neq\emptyset$. Then the corresponding term $D_{\{k|l_k\in\mathbb{Z}_{<0}\}}\, 1_{\gamma}\otimes F_{\{k|r_k\in\mathbb{Z}_{<0}\}}\, 1_{\gamma^{*}}$ is zero which cannot have nonzero contribution.
Thus we need $i\in\{p,q\}$. Similarly, we need $j\in\{p,q\}$. Therefore we need $\{i,j\}=\{p,q\}$, which implies $(3)$. Now we write $\gamma=\gamma_iv_i+\gamma_jv_j$, where $0<\gamma_i,\gamma_j<1$.

\smallskip
	
Now it follows from condition $(2)$ that $-c=l_iv_i+l_jv_j$. Since $l_i\equiv\gamma_i\mod\mathbb{Z}$ and $l_j\equiv\gamma_j\mod\mathbb{Z}$, we have $-c\equiv\gamma\mod\mathbb{Z}v_i+\mathbb{Z}v_j$. This implies $c\equiv\gamma^{*}\mod\mathbb{Z}v_i+\mathbb{Z}v_j$. Then $c=\gamma^{*}+av_i+bv_j=(1-\gamma_i+a)v_i+(1-\gamma_j+b)v_j$ for some $a,b\in\mathbb{Z}$, where $0 < \gamma_i,\gamma_j < 1$. In the definition of the pairing $\langle\Gamma,\Gamma^{\circ}\rangle$, we require that $c$ lies inside the cone $\sigma_{ij}$. So $a,b\in\mathbb{Z}_{\geq0}$. Also since the second coordinates of $c,v_i, v_j$ are all $1$, so we have $a=b=0$. This implies $c=\gamma^{*}$. Since $c+d=v_i+v_j$, we get $d=\gamma$, which implies condition $(4)$.
	

\end{proof}

\begin{proposition}\label{constwist}
The constant contribution of the term in Definition \ref{pairing} without $\Gamma_{(0,0)}$ is
\begin{equation*}
-n\sum_{\substack{\gamma\text{ is a twisted sector }\\ \sigma(\gamma)=\sigma_{ij}}}(j-i)\frac{\sin^{2}(\gamma_i\pi)}{\pi^2}1_{\gamma}\otimes F_{\emptyset,\gamma^{*}}
\end{equation*}
\end{proposition}

\begin{proof}
The constant contribution of the term in Definition \ref{pairing} without $\Gamma_{(0,0)}$ is Equation \eqref{twistno00}.
By Lemma \ref{dualtwist}, the only part of the summation is Equation \eqref{twistno00} that has nonzero contribution to constant is the following
\begin{equation}\label{summationfortwistedsector}
-n\sum_{\{i,j\}\in\Sigma}(j-i)\sum_{\gamma\text{ is a twisted sector in }\sigma_{ij}}x_ix_j\sum_{(l_i),(r_i)}\Gamma_{\gamma^{*}}^{\gamma,(l_i)}\otimes\Gamma_{\gamma}^{\circ,\gamma^{*},(r_i)}.
\end{equation}
We denote by $\varphi(z) = \Gamma(z)^{-1}$ the reciprocal of the $\Gamma$-function.
We get
\begin{equation*}
\begin{split}
x_ix_j\Gamma_{\gamma^{*}}^{\gamma,(l_i)}\otimes\Gamma_{\gamma}^{\circ,\gamma^{*},(r_i)}=&\big(\prod_{k=0}^n(1+\frac{D_k}{2\pi i}\log x_k)(\varphi(1+l_k)+\varphi^{\prime}(1+l_k)\frac{D_k}{2\pi i})1_{\gamma}\big)\\
&\otimes\big(\prod_{k=0}^n(1+\frac{D_k}{2\pi i}\log x_k)(\varphi(1+r_k)+\varphi^{\prime}(1+r_k)\frac{D_k}{2\pi i})F_{\emptyset,\gamma^{*}}\big)\\
=&\prod_{k=0}^n\varphi(1+l_k)\varphi(1+r_k)1_{\gamma}\otimes F_{\emptyset,\gamma^{*}}.
\end{split}
\end{equation*}
By $(2)$ of Lemma \ref{dualtwist}, we have $\sum l_iv_i=l_iv_i+l_jv_j=-c=-\gamma^{*}=(\gamma_i-1)v_i+(\gamma_j-1)v_j$. Since $v_i$ and $v_j$ are independent, we obtain $l_i=\gamma_i-1$ and $l_j=\gamma_j-1$. Hence $r_i=-\gamma_i$ and $r_j=-\gamma_j$. We use the well-known formula $\Gamma(s)\Gamma(1-s)=\frac{\pi}{\sin(s\pi)}$ to get
\begin{equation*}
\begin{split}
\prod_{k=0}^n\varphi(1+l_k)\varphi(1+r_k)=&\varphi(\gamma_i)\varphi(1-\gamma_i)\varphi(\gamma_j)\varphi(1-\gamma_j)\\
=&\frac{\sin(\gamma_i\pi)\sin(\gamma_j\pi)}{\pi^2}=\frac{\sin^{2}(\gamma_i\pi)}{\pi^2}.
\end{split}
\end{equation*}
Therefore, the summation \eqref{summationfortwistedsector} equals to
\begin{equation*}
-n\sum_{\{i,j\}\in\Sigma}(j-i)\sum_{\gamma\text{ is a twisted sector in }\sigma_{ij}}\frac{\sin^{2}(\gamma_i\pi)}{\pi^2}1_{\gamma}\otimes F_{\emptyset,\gamma^{*}},
\end{equation*}
and the result follows.
\end{proof}

\begin{remark}
The proof of $(1)$ in Lemma \ref{dualtwist} implies that contributions of untwisted sector for $\Gamma$ and twisted sector for $\Gamma^{\circ}$ is zero.
\end{remark}

{\bf Step 2.} Now we consider the constant contributed by the terms of the untwisted sector. The terms of the untwisted sector in $x_ix_j\Gamma_c\otimes\Gamma_d^{\circ}$ are
\begin{equation*}
\begin{split}
x_ix_j\Gamma_c^{(0,0)}&\otimes\Gamma_d^{\circ,(0,0)}=\big(x_ix_j\sum_{\substack{\sum l_kv_k=-c \\ \sum r_kv_k=-d \\ l_k,r_k\in\mathbb{Z}}}\prod_{k=0}^n x_k^{l_k+\frac{D_k}{2\pi {\rm i}}}\varphi(1+l_k+\frac{D_k}{2\pi {\rm i}})\, 1_{(0,0)}\big)\otimes\\
&\big(\prod_{k=0}^n x_k^{r_k+\frac{D_k}{2\pi {\rm i}}}\prod_{k,r_k\geq 0}\varphi(1+r_k+\frac{D_k}{2\pi {\rm i}})\prod_{k,r_k< 0}\varphi(1+r_k+\frac{D_k}{2\pi {\rm i}})D_k^{-1}\cdot F_{\sigma}\big).
\end{split}
\end{equation*}

\begin{lemma}\label{conditionuntwisted}
In $x_ix_j\Gamma_c^{(0,0)}\otimes\Gamma_d^{\circ,(0,0)}$, only the terms that belong to the following three cases can have nonzero contribution to the constant.
\begin{itemize}
\item All $l_k=0$, $c=(0,0)$, $r_i=r_j=-1$, all other $r_k=0$;
\item $c=v_i,d=v_j,l_i=r_j=-1$ and all other $l_k,r_k$ are 0;
\item $c=v_j,d=v_i,l_j=r_i=-1$ and all other $l_k,r_k$ are 0.
\end{itemize}
\end{lemma}
\begin{proof}
Only the terms in which $(l_k)$ and $(r_k)$ satisfy
\begin{equation*}
l_k+r_k=0, \text{ for all }k\not=i,j ,\quad l_k+r_k=-1,\text{ if }k=i,j
\end{equation*} can have nonzero contribution to constant.
Moreover, we can ignore the $x_k^{\frac{D_k}{2\pi i}}$ since they will contribute $\log x_k$ which is not a constant. Furthermore, if $l_k<0$ then $\varphi(1+l_k+\frac{D_k}{2\pi i})$ will be a multiple of $D_k$, so we only need to consider terms such that $\{k| l_k<0\}$ has no more than one element. Also, by definition of $H_{(0,0)}^c$, in order to get nonzero contribution, we require that $\sigma$ is either a ray in the interior of $\Sigma$ or a $2-$dimensional cone in $\Sigma$. Hence we only need to consider the following four cases:
\begin{enumerate}
\item $\#\{k,l_k<0\}=0$, $\#\{k,r_k<0\}=1$;
\item $\#\{k,l_k<0\}=0$, $\#\{k,r_k<0\}=2$;
\item $\#\{k,l_k<0\}=1$, $\#\{k,r_k<0\}=1$;
\item $\#\{k,l_k<0\}=1$, $\#\{k,r_k<0\}=2$.
\end{enumerate}

\smallskip

We first show that case $(1)$ cannot happen. Since $l_k=-1-r_k$ for $k=i,j$,  exactly one of $l_k, r_k$ is negative for $k=i,j$. So $\#\{k,l_k<0\}+\#\{k,r_k<0\}\geq 2$.

\smallskip

Now we show that case $(4)$ cannot happen. We consider the following two subcases:

 \smallskip

 One subcase is that the unique $k$ such that $l_k<0$ does not equal $i$ or $j$. Then we have $l_i,l_j\geq 0$, $r_i,r_j<0$, $r_k>0$ and $l_t=r_t=0$ for $t\neq i,j,k$. By the definition of the cohomology $H_{(0,0)}$, we have $D_k\, 1_{(0,0)}=0$ if $\{k\}\not\in\Sigma$. So we need $k\in \Sigma$ to obtain a nonzero term. Since $F_{\sigma}=0$ if $\sigma$ is not a cone in $\Sigma$, we need $\sigma=\{k,r_k<0\}=\{i,j\} \in \Sigma$. Thus $v_i$ and $v_j$ must be consecutive rays in $\Sigma$. This implies $v_k\not\in\sigma_{ij}$. Now by the definition of $L_{c,(0,0)}$, we have $-c=l_iv_i+l_jv_j+l_kv_k$. Thus $-l_kv_k=c+l_iv_i+l_jv_j$. In Definition \ref{pairing}, we know that $c\in\sigma_{ij}$.  So the right hand side of this equation lies in $\sigma_{ij}$, which leads to contradiction.

\smallskip

The second subcase is that $\{k,l_k<0\}$ equals $\{i\}$ or $\{j\}$. Without loss of generality, we assume $i<j$ and $\{k,l_k<0\}=\{i\}$. Now we have $l_j\geq0$. So we can assume $\{t,r_t<0\}=\{j,k\}$ for some $k\neq i,j$. With the same argument as in the first subcase, we need $\{j,k\}\in\Sigma$ and $i\in\Sigma$. Thus $i$ is the smallest one of the three indices $i,j,k$. Now we have $l_i<0,r_i\geq0$, $l_j\geq0,r_j<0$, $l_k>0,r_k<0$ and $l_t=r_t=0$ for $t\neq i,j,k$. So $-c=l_iv_i+l_jv_j+l_kv_k$. Then we consider the equation $c+l_jv_j+l_kv_k=-l_iv_i$. All terms on the left hand side is to the right or on the ray $\mathbb{R}_{\geq0}v_i$. In fact, the left hand side are to the right of the ray $\mathbb{R}_{\geq0}v_i$ since $l_k>0$ and $v_k$ is to the right of $v_i$. This leads to contradiction.

\medskip

 Note that if the total number of negative entries $\#\{k,l_k<0\}+\#\{k,r_k<0\}$ equals  $2$, then all $l_k,r_k$ are zero for all $k\not=i,j$. The reason is that there is exactly one of $l_k, r_k$ is negative for $k=i,j$.

\smallskip


Now we consider case $(2)$. In this case, we have $l_i,l_j\geq0$, $r_i,r_j<0$ and $l_t=r_t=0$ for $t\neq i,j$. We have $-l_iv_i-l_jv_j=c\in\sigma_{ij}$. This implies $l_i=l_j=0$ and $c=(0,0)$. Thus we get $r_i=r_j=-1$, which gives the first case of this lemma.

\smallskip

Then we consider case $(3)$. We first assume $\{k,l_k<0\}=\{i\}$. Now we have $l_i<0,l_j\geq0, r_i\geq0, r_j<0$. Since $\#\{k,l_k<0\}=\#\{k,r_k<0\}=1$, we get $l_k\geq 0$ and $r_k\geq 0$ for $k\neq i,j$. Together by $l_k=-r_k$ for $k\neq i,j$, we get $l_k=r_k=0$ for $k\neq i,j$.
Since $\deg c \leq 1$, we get $c=v_i$, $d=v_j$, and $l_i=r_j=-1$. Similarly, in the case of $\{k,l_k<0\}=\{j\}$, we have $c=v_j$, $d=v_i$, $l_j=r_i=-1$ and $l_k=r_k=0$ for $k\neq i,j$. This gives the second and third cases of this lemma.
\end{proof}

\begin{remark}
In fact, case $(2)$ corresponds to the terms with $\Gamma_{(0,0)}$ in the pairing of Definition \eqref{pairing} and case (3) corresponds to the terms without $\Gamma_{(0,0)}$.
\end{remark}

\smallskip
We are ready to calculate the contribution of the untwisted sector.
We first consider the terms of the case $c=(0,0)$ of Lemma \ref{conditionuntwisted}.
\begin{lemma}\label{(2)}
The constant contributed by the terms $x_ix_j\Gamma^{(0,0)}_{(0,0)}\otimes\Gamma^{\circ,(0,0)}_{v_i+v_j}$ is $\frac{1}{(2\pi{\rm i})^2}1_{(0,0)}\otimes F_{ij}$.
\end{lemma}
\begin{proof}
We have
\begin{equation*}
\begin{split}
x_ix_j\Gamma^{(0,0)}_{(0,0)}\otimes\Gamma^{\circ,(0,0)}_{v_i+v_j}\approx&\big(\prod_{k=0}^n\varphi(1+\frac{D_k}{2\pi {\rm i}})\, 1_{(0,0)}\big)\otimes\big(\prod_{k\not=i,j}\varphi(1+\frac{D_k}{2\pi{\rm i}})(\frac{\varphi(\frac{D_i}{2\pi{\rm i}})}{D_i})(\frac{\varphi(\frac{D_j}{2\pi{\rm i}})}{D_j})F_{ij}\big)\\
\end{split}
\end{equation*}
where $\approx$ means equal after ignoring the nonconstant terms.
Also, we have
\begin{equation*}
\prod_{k=0}^n\varphi(1+\frac{D_k}{2\pi{\rm i}})=\prod_{k=0}^n(1+\varphi^{\prime}(1)\frac{D_k}{2\pi{\rm i}})=1+\frac{\varphi^{\prime}(1)}{2\pi{\rm i}}\sum_{k=0}^nD_k=1
\end{equation*}
and
\begin{equation*}
\prod_{k\not=i,j}\varphi(1+\frac{D_k}{2\pi{\rm i}})(\frac{\varphi(\frac{D_i}{2\pi{\rm i}})}{D_i})(\frac{\varphi(\frac{D_j}{2\pi{\rm i}})}{D_j})F_{ij}=\frac{(\varphi^{\prime}(0))^2}{(2\pi{\rm i})^2}F_{ij}=\frac{1}{(2\pi{\rm i})^2}F_{ij}.
\end{equation*}
Thus, we get
\begin{equation*}
x_ix_j\Gamma^{(0,0)}_{(0,0)}\otimes\Gamma^{\circ,(0,0)}_{v_i+v_j}\approx\frac{1}{(2\pi{\rm i})^2}1_{(0,0)}\otimes F_{ij}.
\end{equation*}
\end{proof}

We now consider the terms $c\neq (0,0)$ of Lemma \ref{conditionuntwisted}.

\begin{lemma}\label{(3)}
The constant contributed by the term $x_ix_j\Gamma_{v_i}^{(0,0)}\otimes\Gamma_{v_j}^{\circ,(0,0)}$ and $x_ix_j\Gamma_{v_j}^{(0,0)}\otimes\Gamma_{v_i}^{\circ,(0,0)}$ are $\frac{1}{(2\pi{\rm i})^2}D_i1_{(0,0)}\otimes F_{j}$ and $\frac{1}{(2\pi{\rm i})^2}D_j1_{(0,0)}\otimes F_{i}$ respectively.
\end{lemma}
\begin{proof}
We have
\begin{equation*}
x_ix_j\Gamma_{v_i}^{(0,0)}\otimes\Gamma_{v_j}^{\circ,(0,0)}\approx\big(\prod_{k\not=i}\varphi(1+\frac{D_k}{2\pi{\rm i}})\varphi(\frac{D_i}{2\pi{\rm i}})\, 1_{(0,0)}\big)\otimes\big(\prod_{k\not=j}\varphi(1+\frac{D_k}{2\pi{\rm i}})\frac{\varphi(\frac{D_j}{2\pi{\rm i}})}{D_j}\, F_{j}\big)
\end{equation*}
The first product is
\begin{equation*}
\begin{split}
\prod_{k\not=i}\varphi(1+\frac{D_k}{2\pi{\rm i}})\varphi(\frac{D_i}{2\pi{\rm i}})=&(1+\frac{\varphi^{\prime}(1)}{2\pi{\rm i}}\sum_{k\not=i}D_k)(\varphi(0)+\varphi^{\prime}(0)\frac{D_i}{2\pi{\rm i}})\\
=&(1-\frac{\varphi^{\prime}(1)}{2\pi{\rm i}}D_i)\frac{D_i}{2\pi{\rm i}}
=\frac{D_i}{2\pi{\rm i}}.
\end{split}
\end{equation*}
The second product is
\begin{equation*}
\begin{split}
\prod_{k\not=j}\varphi(1+\frac{D_k}{2\pi{\rm i}})\frac{\varphi(\frac{D_j}{2\pi{\rm i}})}{D_j}\, F_{j}=&(1-\frac{\varphi^{\prime}(1)}{2\pi{\rm i}}D_j)(\frac{1}{2\pi{\rm i}}+\frac{\varphi^{\prime\prime}(0)}{2}\frac{D_j}{(2\pi{\rm i})^2})F_{j}\\
=&\frac{1}{2\pi{\rm i}}F_{j}+(\frac{\varphi^{\prime\prime}(0)}{2}\frac{1}{(2\pi{\rm i})^2}-\frac{\varphi^{\prime}(1)}{(2\pi{\rm i})^2})D_jF_j.
\end{split}
\end{equation*}
Note that by looking at Taylor expansion of $\varphi$ centered at $0$, we see $\varphi^{\prime\prime}(0)=2\varphi^{\prime}(1)$. Therefore the coefficient of the second term equals to 0.
Then we have
\begin{equation*}
x_ix_j\Gamma_{v_i}^{(0,0)}\otimes\Gamma_{v_j}^{\circ,(0,0)}\approx\frac{1}{(2\pi{\rm i})^2}D_i1_{(0,0)}\otimes F_{j}.
\end{equation*}
Similarly, we have
\begin{equation*}
\begin{split}
x_ix_j\Gamma_{v_j}^{(0,0)}\otimes\Gamma_{v_i}^{\circ,(0,0)}\approx
\frac{1}{(2\pi\rm i)^2}D_j1_{(0,0)}\otimes F_{i}.
\end{split}
\end{equation*}
\end{proof}

\begin{proposition}\label{conuntwist}
The constant contributed by the terms of untwisted sector in the pairing in Definition \ref{pairing} is
\begin{equation*}
\sum_{\substack{i,j\in\Sigma\\ i<j}}(j-i)^2\frac{1}{(2\pi\rm i)^2}1_{(0,0)}\otimes F_{ij}-n\sum_{\substack{i,j\in\Sigma\\ i<j}}(j-i)\frac{1}{2(2\pi\rm i)^2}(D_i1_{(0,0)}\otimes F_{j}+D_j1_{(0,0)}\otimes F_{i}),
\end{equation*}
where $F_0=0$ and $F_n=0$.
\end{proposition}
\begin{proof}
The constant contributed by terms of untwisted sector in the pairing in Definition \ref{pairing} is
\begin{equation}\label{constantuntwistf}
\begin{split}
&\sum_{\substack{i,j\in\Sigma\\ i<j}}(j-i)^2x_ix_j\Gamma_{(0,0)}^{(0,0)}\otimes\Gamma_{v_i+v_j}^{\circ,(0,0)}\\
&-n\sum_{\substack{i,j\in\Sigma\\ i<j}}(j-i)\frac{1}{2}(x_ix_j\Gamma_{v_i}^{(0,0)}\otimes\Gamma_{v_j}^{\circ,(0,0)}+x_ix_j\Gamma_{v_j}^{(0,0)}\otimes\Gamma_{v_i}^{\circ,(0,0)}).
\end{split}
\end{equation}
Then by Lemma \ref{(2)} and Lemma \ref{(3)},  equation \eqref{constantuntwistf} equals to
\begin{equation*}
\sum_{\substack{i,j\in\Sigma\\ i<j}}(j-i)^2\frac{1}{(2\pi\rm i)^2}1_{(0,0)}\otimes F_{ij}-n\sum_{\substack{i,j\in\Sigma\\ i<j}}(j-i)\frac{1}{2(2\pi\rm i)^2}(D_i1_{(0,0)}\otimes F_{j}+D_j1_{(0,0)}\otimes F_{i}).
\end{equation*}
We note that in the second term in the above formula, if $i$ or $j$ lies on the boundary of the fan $\Sigma$, i.e., $i,j \in\{0, n\}$, then $F_i=0$ and $F_j=0$.
\end{proof}

\begin{theorem}\label{pairingforgamma}
We have the following formula.
\begin{equation*}
\begin{split}
\langle\Gamma,\Gamma^{\circ}\rangle=&-n\sum_{\substack{\gamma\text{ is a twisted sector }\\ \sigma(\gamma)=\sigma_{ij}}}(j-i)\frac{\sin^{2}(\gamma_i\pi)}{\pi^2}1_{\gamma}\otimes F_{\emptyset,\gamma^{*}}\\
&+\sum_{\substack{i,j\in\Sigma\\ i<j}}(j-i)^2\frac{1}{(2\pi\rm i)^2}1_{(0,0)}\otimes F_{ij}\\
&-n\sum_{\substack{i,j\in\Sigma\\ i<j}}(j-i)\frac{1}{2(2\pi\rm i)^2}(D_i1_{(0,0)}\otimes F_{j}+D_j1_{(0,0)}\otimes F_{i})
\end{split}
\end{equation*}
where $F_0=0$ and $F_n=0$.
\end{theorem}
\begin{proof}
Follows from Proposition \ref{constwist} and Proposition \ref{conuntwist}.
\end{proof}

\smallskip
\begin{theorem}\label{mainfor(4)}
	We have
	\begin{equation*}
	-\frac{n}{4\pi^2}\chi_{H}^{-1}=\langle\Gamma,\Gamma^{\circ}\rangle.
	\end{equation*}
\end{theorem}
\begin{proof}
We compare the result of Theorem \ref{pairingforgamma} with that of Theorem \ref{chiH-1}, while using Proposition \ref{F_ij}
$$\langle\Gamma,\Gamma^{\circ}\rangle_{\text{twisted}}=-n\sum_{\substack{\gamma\text{ is a twisted sector }\\ \sigma(\gamma)=\sigma_{ij}}}(j-i)\frac{\sin^{2}(\gamma_i\pi)}{\pi^2}1_{\gamma}\otimes F_{\emptyset,\gamma^{*}}=-\frac{n}{4\pi^2}(\chi_{H}^{-1})_{\text{twisted}}.$$
The part contributed by pairing of degree $0$ with degree $2$ terms in $\langle\Gamma,\Gamma^{\circ}\rangle_{(0,0)}$ is
\begin{equation*}
	\sum_{\substack{i,j\in\Sigma\\ i<j}}(j-i)^2\frac{1}{(2\pi\rm i)^2}1_{(0,0)}\otimes F_{ij}=\sum_{\substack{i,j\in\Sigma\\ i<j}}(j-i)\frac{1}{(2\pi\rm i)^2}1_{(0,0)}\otimes F=-\frac{n}{4\pi^2}1_{(0,0)}\otimes F.
\end{equation*}
The part contributed by pairing of degree $1$ with degree $1$ terms in $\langle\Gamma,\Gamma^{\circ}\rangle_{(0,0)}$ is
\begin{equation}\label{8pi}
\begin{split}
&\frac{n}{8\pi^2}\sum_{\substack{i,j\in \Sigma \\0\leq i<j\leq n}}(j-i)(D_i\otimes F_j+D_j\otimes F_i)\\
=&\frac{n}{8\pi^2}\sum_{\substack{j\in \Sigma \\0<j<n}}(jD_0+(n-j)D_n)\otimes F_j+\frac{n}{8\pi^2}\sum_{\substack{i,j\in \Sigma \\0<i,j<n}}|j-i|D_i\otimes F_j.
\end{split}
\end{equation}
After replacing $D_0$ and $D_n$ by their expressions in terms of $D_i$ where $i$ lies in the interior of $\Sigma$,  equation \eqref{8pi} equals
\begin{equation*}
\begin{split}
\frac{n}{8\pi^2}\sum_{\substack{0<i,j<n \\ i,j\in\Sigma}}(\frac{2ij-j-in}{n}+|j-i|)D_i\otimes F_j
\\
=\frac{n}{8\pi^2}\sum_{0<k,l<r}(\frac{2i_ki_l-i_l-i_kn}{n}+|i_l-i_k|)D_{i_k}\otimes F_{i_l}.
\end{split}
\end{equation*}
It is easy to check the coefficient equals $-\frac{n}{4\pi^2}(-g_{kl})$.
	Comparing the formulas in Theorem \ref{pairingforgamma} and Theorem \ref{chiH-1}, we get our result.
\end{proof}

\begin{corollary}
Conjecture \ref{7.3}
holds for  $\mathrm{rk} \, \mathrm{N}=2$.
\end{corollary}

\medskip

In \cite{BorisovHorja}, the Gamma series solutions of Definition \ref{gammac} were expected to be compatible with pullback-pushforward and analytic continuation. Let $\Sigma_1$ and $\Sigma_2$ be two adjacent triangulations. Then the stacks $\mathbb{P}_{\Sigma_1}$ and $\mathbb{P}_{\Sigma_2}$ differ by a flop that is a composition of weighted blowup and weighted blowdown. It is implied by Theorem $4.2$ of \cite{Ka1} that the pullback-pushforward functors $D(\mathbb{P}_{\Sigma_1}) \rightarrow D(\mathbb{P}_{\Sigma_2})$ and $D^c(\mathbb{P}_{\Sigma_1}) \rightarrow D^c(\mathbb{P}_{\Sigma_2})$ are equivalences. Thus there are induced natural group isomorphisms $pp:K_0(\mathbb{P}_{\Sigma_1}) \rightarrow K_0(\mathbb{P}_{\Sigma_2})$ and $pp:K^c_0(\mathbb{P}_{\Sigma_1}) \rightarrow K^c_0(\mathbb{P}_{\Sigma_2})$. The Conjecture \ref{7.1} was proposed in \cite{BorisovHorja} to describe the compatibility between the Gamma series solutions and  pullback-pushforward and analytic continuation. We are now ready to prove it in dimension two.

\begin{theorem}
Conjecture \ref{7.1} holds for $\mathrm{rk} \, \mathrm{N}=2$.
\end{theorem}
\begin{proof}
It suffices to consider the case $v_0=(0,1),\ldots,v_n=(n,1)$, because the others are obtained by setting some variables to zero. In this case, bbGKZ is the usual GKZ. The top diagram is true by previous work \cite{BorisovHorja2}. Also, our pairing in Definition \ref{pairing} has no monodromy. Then the bottom diagram is true by duality.
\end{proof}

\section{Future directions}\label{futuredire}
We expect to extend this calculation to higher dimensions.
Moreover, we would like to be able to categorify the bbGKZ systems to obtain two
families of triangulated categories over the parameter space of $x$ such that in the neighborhoods of toric degeneracy points the Gamma series provide some kind of character map to the corresponding $K-$theory.

\end{document}